\documentclass[a4paper,12pt]{amsart}
\usepackage{vmargin}
\usepackage[pdftex]{graphicx}
\usepackage{amsfonts}
\usepackage{amssymb}
\usepackage{mathrsfs}
\usepackage{stmaryrd}
\usepackage{amsmath}
\usepackage{amsthm}
\usepackage{enumerate}
\usepackage{nomencl}
\usepackage[bookmarks=true]{hyperref}   % Hyperref in DVI and PDF
\usepackage{esint}
\usepackage{dsfont}
\usepackage{upgreek}
\usepackage{color}
\usepackage{comment}
\usepackage{float}

\makeatletter
\renewcommand\section{\@startsection{section}{1}{\z@}%
                       {-3\p@ \@plus -4\p@ \@minus -4\p@}%
                       {3\p@ \@plus 4\p@ \@minus 4\p@}%
                      {\normalfont\normalsize\centering\scshape}}
\makeatother

\hypersetup{
    colorlinks,%
}

\author{Lashi Bandara}
\title{Rough metrics on manifolds and quadratic estimates}
\date{\today}
\address{Lashi Bandara, Centre for Mathematics and its Applications, 
Australian National University, Canberra, ACT, 0200, Australia}
\urladdr{\href{http://maths.anu.edu.au/~bandara}{http://maths.anu.edu.au/~bandara}}
\email{\href{mailto:lashi.bandara@anu.edu.au}{lashi.bandara@anu.edu.au}}

\def\colour{\colour}

\def\colour{\color}

\newtheorem{theorem}{Theorem}[section]
\newtheorem{corollary}[theorem]{Corollary}
\newtheorem{lemma}[theorem]{Lemma}
\newtheorem{proposition}[theorem]{Proposition}

\newtheorem{definition}[theorem]{Definition}
\newtheorem{remark}[theorem]{Remark}

\newtheorem{example}[theorem]{Example}

\newcommand{\mdot}{\cdotp}

\newcommand{\cbrac}[1]{\left(#1\right)}

\newcommand{\dbrac}[1]{\left\{#1\right\}}
\newcommand{\modulus}[1]{|#1|}
\newcommand{\set}[1]{\dbrac{#1}}

\newcommand{\dom}{ {\mathcal{D}}}
\newcommand{\ran}{ {\mathcal{R}}}
\newcommand{\nul}{ {\mathcal{N}}}

\newcommand{\comp}{\circ}

\newcommand{\R}{\mathbb{R}}
\newcommand{\C}{\mathbb{C}}

\newcommand{\script}[1]{\mathscr{#1}}

\DeclareMathOperator{\re}{Re}			% Real part
\renewcommand{\emptyset}{\varnothing}
\newcommand{\union}{\cup}

\newcommand{\intersect}{\cap}

\newcommand{\rest}[1]{{{\lvert_{}}_{}}_{#1}}
\newcommand{\close}[1]{\overline{#1}}		% closure
\DeclareMathOperator{\img}{Img}

\newcommand{\ind}[1]{\raisebox{\depth}{\(\chi\)}_{#1}}	% Indicator funciton

\newcommand{\Char}[1]{\chi_{#1}} 	% Characteristic function 

\renewcommand{\epsilon}{\varepsilon}

\renewcommand{\phi}{\varphi}

\DeclareMathOperator{\Graph}{graph}

\newcommand{\tp}[1]{{#1}^{\mathrm{tr}}}
\newcommand{\tensor}{\otimes}
\newcommand{\norm}[1]{\| #1 \|}			% Norm
\newcommand{\spt}[1]{{\rm spt} {\text{ }}#1}	% Support 

\DeclareMathOperator{\essinf}{essinf}
\newcommand{\interior}[1]{\mathring{#1}}	% Interior

\DeclareMathOperator{\tr}{tr}			% Trace
\DeclareMathOperator{\len}{\ell}			% Length
\DeclareMathOperator{\divv}{div}		% Divergence

\newcommand{\cut}{\ \llcorner\ }			% cut product

\newcommand{\partt}[1][{}]{{\partial_{{#1}}}}		% Partial

\newcommand{\Rm}{\rm{Rm}}			% Reimannian Curvature Op
\newcommand{\Ric}{{\rm Ric}}			% Ricci Curvature
\DeclareMathOperator{\inj}{inj} 		% injectivity radius

\newcommand{\bnd}{\partial}			% Boundary partial d
\newcommand{\Forms}[1][{}]{\mathbf{\Omega}^{#1}}		% P-Forms\\
\newcommand{\Tensors}[1][{}]{{\mathcal{T}}^{(#1)}}	% Tensors

\newcommand{\Sect}{\mathbf{\Gamma}}		% Sections (tensor feilds)
\newcommand{\tanb}{{\rm T}}		% Tangent Bundle
\newcommand{\cotanb}{{\rm T}^\ast}	% Cotangent bundle
\newcommand{\pushf}[1]{{#1}_\ast}			% Push forward
\newcommand{\pullb}[1]{{#1}^\ast}			% Pull back

\DeclareFontFamily{OT1}{restrictfont}{}
\DeclareFontShape{OT1}{restrictfont}{m}{n}{<-> fmvr8x}{}

\newcommand{\adj}[1]{{#1}^\ast}			% Adjoint star
\newcommand{\extd}{{\rm d}}			% Exterior Derivative
\newcommand{\intd}{{\updelta}}
\newcommand{\Dir}{{\rm D}}			% Dirac

\newcommand{\inprod}[1]{\langle #1 \rangle}	% inner product braces
\newcommand{\grad}{\nabla}			% Gradient delta
\newcommand{\conn}[1][{}]{{\grad_{{#1}}}}		% Connection
\DeclareMathOperator{\Lip}{\bf Lip}			% Lip
\newcommand{\Leb}[1][{}]{\script{L}^{#1}}			% Lebesgue L

\newcommand{\bddlf}{\mathcal{L}} 	% Bounded Linear Functions over #1
\newcommand{\spec}{\sigma}		% Spectrum of #1
\newcommand{\conj}[1]{\overline{#1}}				% complex conjugate
\DeclareMathOperator{\nr}{nr}				% Numerical Range
\newcommand{\Lp}[2][{}]{{\rm L}^{#2}_{\rm #1}}		% L_{#1}^{#2}
\newcommand{\Ck}[2][{}]{{\rm C}^{#2}_{\rm #1}}		% C^k_c
\newcommand{\Lips}[1][{}]{{\rm Lip}_{\rm #1}}		% C^k_c
\newcommand{\Sob}[2][{}]{{\rm W}^{#2}_{\rm #1}}		% W^{k,p}_{c} Sobolev space 
\newcommand{\SobH}[2][{}]{{\Sob[#1]{#2,2}}}	% H^k_{c}  Sobolev space 

\newcommand{\convolve}{\, \ast\, }

\newcommand{\iden}{{\mathrm{I}}}

\newcommand{\Hil}{\script{H}}			% Hilbert Space
\newcommand{\Rend}{{\rm R}}

\newcommand{\sC}{\script{C}}

\newcommand{\sR}{\script{R}} 

\newcommand{\cC}{\mathcal{C}}

\newcommand{\cV}{\mathcal{V}}
\newcommand{\cM}{\mathcal{M}} 
\newcommand{\cN}{\mathcal{N}}
\newcommand{\cP}{\mathcal{P}}

\newcommand{\cT}{\mathcal{T}}

\newcommand{\Spa}{\mathcal{X}}

\newcommand{\mg}{\mathrm{g}}
\newcommand{\mgt}{{\tilde{\mg}}}
\newcommand{\mh}{\mathrm{h}}

\newcommand{\met}{\uprho} 
\newcommand{\Sph}{\mathrm{S}}

\newcommand{\Rosen}{Ros\'en} 
\newcommand{\B}{\mathrm{B}}
\newcommand{\E}{\mathrm{E}}

\DeclareMathOperator{\Sing}{Sing}
\DeclareMathOperator{\Reg}{Reg}

\newcommand{\RNum}[1]{\uppercase\expandafter{\romannumeral #1\relax}}

\begin{document}

\maketitle

%\vspace*{-2em}
\begin{abstract}
We study the persistence of quadratic
estimates related to the Kato square
root problem across a change of metric
on smooth manifolds by 
defining a class of Riemannian-like
metrics that are permitted to be of low regularity
and degenerate on sets of measure
zero. We also demonstrate how to
transmit quadratic estimates between 
manifolds which are homeomorphic and
locally bi-Lipschitz.
As a consequence,
we demonstrate the invariance
of the Kato square root problem under Lipschitz 
transformations of the space
and obtain solutions to this
problem on functions
and forms on compact
manifolds with a continuous metric.
Furthermore, we show that a
lower bound on the injectivity radius is
not a necessary condition to solve the
Kato square root problem. 
\end{abstract}
\vspace*{-0.5em}
\tableofcontents
\vspace*{-2em}

\parindent0cm
\setlength{\parskip}{\baselineskip}

%\vspace{\sectionendspace}
\section{Introduction}
\label{Sect:Intro}

Quadratic estimates are instrumental
in the study of functions of bi-sectorial operators
as these estimates are equivalent 
to the existence of a  bounded holomorphic functional calculus.
These operators capture an 
important class of partial differential 
operators 
and quadratic estimates
provide a quantitative mechanism by which to
analyse them.
The survey papers \cite{ADMc} by Albrecht, Duong, and McIntosh,
\cite{AAMc} by Auscher, Axelsson (\Rosen), and McIntosh,
and \cite{HMc} by Hofmann and McIntosh
are excellent sources which illustrate
the virtues of quadratic estimates
and their application to partial differential 
equations.
%\Rd [REFERENCES?]\Bk\ 
%Quadratic estimates may be proved using different methods
%but we are most familiar with techniques arising from harmonic analysis. 
%This is, in essence, the
%general connection between functional calculi and harmonic analysis.

In this paper, we will be concerned with quadratic estimates
associated to the Kato square root problem, which is the
problem of determining the domains of
square roots of uniformly elliptic second-order divergence-form 
differential operators. 
This problem on $\R^n$ was
solved in 2002 by Auscher, Hofmann, Lacey, McIntosh and Tchamitchian in
\cite{AHLMcT}. This result was recaptured in a 
first-order framework  in 2005 by Axelsson (\Rosen), Keith, and McIntosh in
 \cite{AKMc}. 
In the same paper, the authors also 
solve the Kato square root problem
for differential forms on $\R^n$, as well as
similar problems on compact manifolds.

The AKM approach has been successfully adapted to solve a range of 
problems since its conception. In \cite{AKM2}, the same authors
tackle the Kato square root problem for domains with mixed boundary values,
in \cite{Morris3} Morris adapts this framework to solve the Kato square root problem
for Euclidean submanifolds with second fundamental form bounds, 
and in \cite{BMc}, McIntosh and the author
solve Kato square root problems on manifolds 
under mild assumptions on the geometry.
This AKM technology  has also been applied to the
study of square roots of sub-elliptic operators
on Lie groups by ter Elst, McIntosh and the author 
in \cite{BEMc}.  More recently, this framework has been 
adapted by Leopardi and Stern to study numerical problems 
in \cite{LS}. 

The central theme of this paper is to reveal an
important connection between the study of quadratic
estimates and the study of Kato square root problems
on manifolds with \emph{non-smooth} metrics.
Our philosophy is to encode the lack of regularity
of the ``rough'' metric into a rough coefficient operator 
on a nearby smooth metric.
We emphasise that the first-order perspective of 
the AKM framework is of paramount importance 
in our work.
 
We remark that this philosophy 
is not necessarily our own insight and that we have borrowed
it in large parts from the authors of \cite{AKMc}. They dedicate an entire
section in their paper to deal with the
\emph{holomorphic dependency} of the functional calculus
and show Lipschitz estimates for the calculus
when the metric on a compact manifold is perturbed
slightly. The main novelty here
is that we allow our metrics
to be of very low regularity and even degenerate
on null measure sets.
 
We also derive motivation from \cite{AKM2},
where the authors
exploit the invariance of quadratic estimates under
Lipschitz transforms to solve Kato square root problems on Lipschitz
domains. However,
since a metric on a manifold is the key global 
object which determines the underlying geometry, 
we concern ourselves with the study of
non-smooth metrics, rather the special case of such transforms.
 
To increase the readability of this paper,
we give an overview of our achievements 
section by section. 
In \S\ref{Sect:Prelim}, we introduce our notation
and provide an overview of 
the aspects of manifolds which are
of importance to us, including their Lebesgue space theory.
In particular, we outline 
the measure theoretic notions that will be
of use to us in order to talk about degeneracy 
and low regularity in subsequent sections. 
For the benefit of the reader, we also 
present an overview of the AKM framework.

Following onto \S\ref{Sect:Rough},
we formulate a notion of a 
\emph{rough metric} (Definition
\ref{Def:RoughMet}), which is a Riemannian-like
metric that may not only be low in its regularity, but which
may degenerate on large, non-trivial, but null measure sets.
We also consider Lebesgue and Sobolev space
theory for such metrics.

Then, in \S\ref{Sect:Red}, 
we demonstrate how to reduce Kato square root problems
on rough metrics to similar problems on a \emph{nearby} smooth 
metric. A perspective that emerges from our
developments in this section is that the Kato square root problems
which concern us are dependent more
on the measure than the induced distance metric. We also use the
technology we build to show that the Kato square root problem 
on compact manifolds can always be solved in the affirmative 
when the metric is
only assumed to be continuous (Theorem \ref{Thm:CpctContMet}).

We dedicate \S\ref{Sect:QuadIsom}
to the study of transmitting quadratic
estimates between two manifolds.
To make this study sufficiently general and non-trivial,
we assume that the map between these two manifolds is
a homeomorphism that is locally bi-Lipschitz. 
We show that if one manifold has a smooth metric and 
admits quadratic estimates, then so does the other
in the induced, low regularity pullback metric (Theorem \ref{Thm:QuadEstLipeo}).
As a consequence, we obtain that the 
Kato square root problem for functions 
is invariant under Lipschitz transformations
of the geometry (Theorem \ref{Thm:KatoLipInv}).

Lastly, in \S\ref{Sect:Inj}, we consider
the question posed by McIntosh 
as to whether a lower bound on the injectivity 
radius is necessary to solve
Kato square root problems. 
This question is motivated by the fact that, 
in every setting where this problem is solved, 
lower bounds on injectivity radius
are present. The most geometrically general
such theorem (in \cite{BMc}) 
reveals how this technicality assists in the proof.
In this section, we answer this question in the negative
by demonstrating the existence of smooth metrics on $\R^2$
with zero injectivity radius 
arbitrarily close to the Euclidean 
metric. We then apply the tools we have 
developed in \S\ref{Sect:Red} 
to perturb the Euclidean solution to a solution 
on the nearby metric (Theorem \ref{Thm:KatoInj2dim}).
More seriously, we also show 
that such solutions are abundant
in all dimensions greater than two
by constructing metrics with
zero injectivity radius
that are smooth everywhere
but a point and which are arbitrarily close to any metric
for which we can solve the Kato square root problem
(Theorem \ref{Thm:KatoInj}). 
\section*{Acknowledgements}

This work was made possible
through the support of
the Australian-American Fulbright Commission
through a Fulbright Scholarship,
the Australian National University
by the way of a supplementary PhD scholarship, 
and by the Australian Government through an 
Australian Postgraduate Award and
an Australian Research Council grant of
Alan McIntosh and Pierre Portal.
The author wishes to acknowledge the support of these
organisations as well as both Alan and Pierre for
their continuing support.

Steve Zelditch deserves a special mention as this 
paper was motivated by a conversation with him. 
The author also wishes to 
acknowledge Rick Schoen, his Stanford Fulbright mentor,
for his useful insights and feedback. Also,
Kyler Siegel, Nick Haber, Boris Hanin,
Mike Munn, Alex Amenta, Ziping Rao and Travis Willse deserve a 
mention for indulging the author in invigorating conversations.
 
The first and less ambitious version of this 
paper was scrapped and re-written after 
a conversation with Sebastien Lucie.
Annegret Burtscher deserves a special 
mention for her helpful insights
and for offering corrections.
Also, Anton Petrunin,  Sergei V. Ivanov and 
Otis Chodosh need to be acknowledged
for their insights, comments
and contributions to the final section 
of this paper.
\section{Preliminaries} 
\label{Sect:Prelim}

\subsection{Notation}
Throughout this paper, we assume the Einstein 
summation convention. That is, whenever a
raised index appears against a lowered index,
we sum over that index, unless specified otherwise.

For a function or section $f$ mapping into a
topological vector space, we denote its \emph{support},
by $\spt f = \close{\set{x : f(x) \neq 0}}$, 
the closure of the set of points that are non-zero under $f$. 

By $\Lp{p}$ we denote the $\Lp{p}$ spaces 
in the given context. Typically, we will write
$\Lp{p}(\cV,\mh)$ to denote the
$\Lp{p}$ space over the vector bundle 
$\cV$ with metric $\mh$.
In writing $\Ck{k,\alpha}$
we mean $k$ times differentiable objects
that are $\alpha$-H\"older continuous
in the $k$-th derivative. When $\alpha = 0$, we
simply write $\Ck{k}$.
The subspace of objects with compact support are denoted
by $\Ck[c]{k,\alpha}$ and $\Ck[c]{k}$ respectively.

We reserve the symbol $\Hil$ to denote a 
\emph{Hilbert space}. 
By an operator $\Gamma: \Hil \to \Hil$, we mean 
a linear map defined on a subset $\dom(\Gamma) \subset \Hil$
which is the \emph{domain} of the operator.
The \emph{null space} and \emph{range} of this operator
are denoted by $\nul(\Gamma)$ and $\ran(\Gamma)$
respectively. The algebra of bounded
linear operators are then given by $\bddlf(\Hil)$.

In our analysis, we often write $a \lesssim b$
to mean that $a \leq C b$, where $C$ is some constant. 
The dependencies of $C$ will either be explicitly
specified or otherwise, clear from context.
By $a \simeq b$ we mean that $a \lesssim b$ and
$b \lesssim a$. 

Unless specified otherwise, we will reserve 
the symbol $\delta$ to denote the standard
inner product on $\R^n$.
The Lebesgue measure on $\R^n$ will be
denoted by $\Leb$.

\subsection{Manifolds and their Lebesgue spaces}
\label{Sect:Prelim:Mfld}
The fundamental objects that lie at the heart 
of our study in this paper  
will consist of a pair $(\cM, \mg)$ where $\cM$
is \emph{manifold} and $\mg$ is a \emph{metric} on $\cM$.
By an $n$-manifold $\cM$, we will always mean a topological 
space that is \emph{second-countable}, 
\emph{Hausdorff}, and \emph{locally
homeomorphic} to open subsets of $\R^n$.
These homeomorphisms are \emph{coordinate charts}
on $\cM$ and their collection is called
an \emph{atlas} or \emph{differentiable structure}. If the 
homeomorphisms are $\Ck{k,\alpha}$-diffeomorphisms,
by which we mean that compositions of one chart
with the inverse of another (as long as their domains have 
nonempty intersection) is a $\Ck{k,\alpha}$-diffeomorphism,
then we call $\cM$ a $\Ck{k,\alpha}$ manifold.
A $\Ck{0}$ manifold is called a \emph{topological} manifold, 
a manifold that is $\Ck{0,1}$ is called Lipschitz,
and $\Ck{\infty}$ manifold is said to be \emph{smooth}. 
The \emph{geometry} of a manifold is dictated
by the metric and it is to metrics that we
shall devote significant attention throughout
this paper.

Our wish is to study geometries that contain \emph{singularities}.
The word `singularity' is used to mean many different
things, but to us, it can mean one of two things:
either singularities arising from within 
the differentiable structure of $\cM$,
or singularities arising from
a lack of regularity in $\mg$.

If we were to pursue the former notion,
the most singular
structure we can consider is $\Ck{0}$. 
It is difficult to envisage setting up 
analysis in this situation.
However, due to 
a theorem of Sullivan (see \cite{Sullivan}), 
it turns out that
in all dimensions but $4$, 
every topological manifold can be
made into a Lipschitz manifold!
This does not improve our situation much - simply 
because changes of coordinates for vectorfields
locally involves derivatives of charts, 
and if the charts are at best bi-Lipschitz,
then the vectorfields can be at most
measurable in regularity.
Consequently, analysis on Lipschitz manifolds
is both difficult and limited,
but remarkably, Gol'dshtein, Mitrea
and Mitrea makes considerable progress
setting up analysis on compact Lipschitz manifolds
in \cite{GMM}.

Given what we have just said,
and since we would like to measure
varying degrees of regularity beyond
simply Lipschitz or measurable in order 
to reduce
low regularity problems to smooth ones, 
we are forced to confine our attention 
to higher regularity
differentiable structures.
It is reasonable to expect that 
we will have to consider the regularity
classes $\Ck{k,\alpha}$ ($k \geq 1,\ \alpha \in [0,1]$)
separately. However, 
a classical result of Whitney (see \cite{Whitney})
states  that every  $\Ck{1}$  structure
(and hence, every $\Ck{k,\alpha}$ ($k\geq 1$) manifold)
can be smoothed. 
This exactly means that the manifold can be given 
a smooth atlas compatible with the initial $\Ck{k,\alpha}$ ($k\geq 1)$ 
atlas.
Let us remark that this this is false for topological manifolds 
by a counterexample due to Kervaire in \cite{Kervaire}.
Thus, from this point onwards, without the loss of generality,  
we assume that 
$\cM$ is smooth, and our efforts are concentrated
on studying situations where the metric $\mg$ is
typically of lower regularity.

Recall that for a $\Ck{k}$ ($k\geq 0$) metric $\mg$,
the \emph{length functional} of an
absolutely continuous curve $\gamma:I \to \cM$
is given by 
$$\len_\mg(\gamma) = \int_{I} \modulus{\dot{\gamma}(t)}_{\mg(\gamma(t))}\ dt,$$
where $\modulus{u}_{\mg(x)} = \sqrt{\mg_x(u,u)}.$
The induced distance is given by 
$$\met_\mg(x,y) = \inf\set{\len_\mg(\gamma): 
\gamma(0) = x,\ \gamma(1) = y,\ \gamma\  \text{abs. cts.}}.$$
We remark that for $\Ck{k}$ metrics with $k \geq 0$, 
Burtscher demonstrates 
through Theorem 3.11 and Corollary 3.13 in her paper \cite{Burtscher}
that one can equivalently consider
curves that are piecewise smooth.

A curve $\gamma:I \to \cM$ is said to be a \emph{geodesic}
for $\met_\mg$ if for each $t \in I$, 
there is a sub-interval $J_t \subset I$ with $t \in J_t$
such that for every $t_1, t_2 \in J_t$,
$\met_\mg(\gamma(t_1, t_2)) = \modulus{t_1 - t_2}$.
This is a metric-space formulation of geodesy.
A more geometric characterisation can be given 
to metrics $\mg$ of class $\Ck{2}$ or higher.
Recall, that in this situation there exists
the Levi-Cevita connection $\conn$ with respect to $\mg$.
Then, $\gamma$ is a geodesic 
if and only
$\conn[\dot{\gamma}]{\dot{\gamma}} = 0$. 
The latter notion 
captures that the curve is lacking tangential acceleration 
and hence, is in free-fall. 
A curve $\gamma$ is said to be a \emph{minimising geodesic}
if $J_t = I$. It is easy to see that every geodesic is \emph{locally minimising}, 
and that a minimising geodesic realises the 
\emph{shortest distance} between the end points of the curve.

For a metric that is $\Ck{2}$ at a point $x \in \cM$,
the exponential map $\exp_x:V_x \subset T_x\cM \to \cM$
takes a velocity vector $v \in V_x$ as input and
outputs the end point of the unique unit-length 
parametrised geodesic $\gamma_v$ of velocity $v$.
For each such point $x$, there is an $r_x > 0$
such that $\exp_x: B_{r_x}(0) \subset T_x\cM \to \cM$
is a homeomorphism with its image. 
The largest such radius $r_x$ is then the injectivity 
radius $\inj(\cM,\mg, x)$ at the point $x \in \cM$.
The injectivity radius over the entire manifold
is then given by $\inj(\cM,\mg) = \inf_{x \in \cM} \inj(\cM,\mg,x)$.

A metric $\mg$ also induces a canonical volume
measure on $\cM$. It is constructed by pasting together
the local expression 
$$d\mu_\mg(x) = \sqrt{\det \mg(x)}\ d\Leb(x),$$
via a smooth partition of unity
subordinate to a covering of $\cM$ by charts.
This construction is readily checked to be 
well-defined.

In the analysis on manifolds with non-smooth metrics, 
it will be convenient for us to 
have certain measure-theoretic notions that can be 
formulated independently of a metric. In this 
spirit, we define the following. 
\begin{definition}[Notions of measure]
We say that: 
\begin{enumerate}[(i)] 
\item a set $A \subset \cM$ is measurable
if whenever $(U,\psi)$ is a chart satisfying $U \intersect A \neq \emptyset$,
then $\phi(U\intersect A) \subset \R^n$ is $\Leb$-measurable, 

\item a function $f: \cM \to \C$ is measurable if
$f \comp \psi^{-1}: \psi(U) \to \C$ is $\Leb$-measurable
	for each chart $(U,\psi)$, 

\item a tensor field $T: \cM \to \Tensors[r,s]\cM$ is measurable
	if the coefficients $T_{i_1, \dots, i_r}^{j_1, \dots, j_s}$ 
	in each $(U,\psi)$ is measurable

\item a set $Z$ is a \emph{null set} or \emph{set of null measure} 
	if requiring $\Leb(\phi(U \intersect Z)) = 0$ for each chart $(U, \psi)$, 
\item a property $P$ is valid almost-everywhere 
if it is valid $\Leb$-a.e. in each coordinate chart $(U,\psi)$.
\end{enumerate}
\end{definition}

In what is to follow,
we always identify the $(r,s)$-tensors, the tensors
of covariant rank $r$ and contravariant rank $s$, as
a complex vector bundle after complexification 
of its real counterpart. The set of 
\emph{measurable sections} for such tensors are then denoted by
$\Sect(\Tensors[a,b]\cM)$.

While we have formulated these notions independently 
of a metric, we can indeed recapture these definitions
in terms of a background metric. This yields a global, geometric
and coordinate independent perspective.
But first, we require the following lemma concerning
the preservation of measurability. 
We phrase this result in a far more general context than
we need here because this generality will be of use to us later.
Note that in the following proposition, we take 
the measure $\sigma$-algebra to be the maximal one. 
That is, a set $M \subset \Spa$ is $\mu$-measurable
if and only for \emph{every} set $A \subset \Spa$, 
$\mu(A) = \mu(A\setminus M) + \mu(A \intersect M)$.

\begin{lemma}
%\label{Prop:WeightedMeas}
\label{Lem:WeightedMeas}
Let $(\Spa, \nu)$ be a measure-space
and let $f: \Spa \to [0,\infty]$ be a
$\nu$-measurable function such that $0 < f(x) < \infty$
for $x$ $\nu$-a.e. Suppose that $d\mu(x) = f(x) d\nu(x)$
for $x$ $\nu$-a.e. Then, a set $M \subset \Spa$
is $\nu$-measurable if and only if it is $\mu$-measurable
and $\mu$ and $\nu$ share the same sets of measure zero.
\end{lemma}
\begin{proof}
Suppose that $M \subset \Spa$ is $\nu$-measurable
and $A \subset \Spa$ is any other set.
We show that
$\mu(A) = \mu(A \intersect M) + \mu(A \setminus M)$.

First, assume that $f$ is a simple function
and let $f = \sum_{k=1}^m \Char{F_k}$.
Then, 
$$\mu(A) = \int_{A} \sum_{k=1}^m \Char{F_k}\ d\nu(x)
	= \sum_{k=1}^m \nu(A \intersect F_k).$$
Similarly, $\mu(A \intersect M)  = \sum_{k=1}^m \nu(A \intersect F_k \intersect M)$
and since $(A \setminus M) \intersect F_k = (A \intersect F_k) \setminus M$, 
$\mu(A \setminus M) = \sum_{k=1}^m \nu((A \intersect F_k))\setminus M)$.
By the $\nu$-measurability of $M$,
$\nu(A \intersect F_k) = \nu(A \intersect F_k \intersect M) + \nu((A \intersect F_k )\setminus M)$
which shows that $\mu(A) = \mu(A\intersect M) + \mu(A \setminus M)$.

Now, since $f$ is assumed to be non-negative and $\nu$-measurable,
there exists a sequence of simple functions $f_n$ monotonically 
increasing to $f$. Thus, by the monotone convergence theorem, 
$$\mu(A) = \int_{A} f\ d\nu 
	= \lim_{n\to\infty} \int_{A} f_n\ d\nu$$
and by what we have established previously, 
$\int_{A} f_n\ d\nu = \int_{A\intersect M} f_n\ d\nu + \int_{A \setminus M} f_n\ d\nu$.
Therefore,
$\mu(A) = \mu(A \intersect M) + \mu(A \setminus M)$.
Furthermore, it is easy to see that $\nu(Z) = 0$
implies $\mu(Z) = 0$.

Next, note that
$$\cbrac{\frac{1}{f}}^{-1}(\alpha, \infty] 
	= \set{x \in \Spa: \frac{1}{f(x)} > \alpha}
	= \set{x \in \Spa: f(x) < \alpha} 
	= f^{-1}[-\infty,\alpha).$$
So $1/f$ is $\nu$-measurable which, by what
we have just established, $\mu$-measurable.
Also since $0 < f < \infty$ $\nu$-a.e.,
which implies $\mu$-a.e., we can write
$d\nu(x) = 1/f(x)\ d\mu(x)$
for $x$ $\mu$-a.e. Thus, by repeating the 
previous argument with $\nu$ and $\mu$ 
interchanged, we obtain that a $\mu$-measurable
set $M$ is $\nu$-measurable and
that whenever $\mu(Z) = 0$ implies
$\nu(Z) = 0$.
\end{proof}

This lemma then allows us to prove the following
geometric rephrasing of our measure notions.
\begin{proposition}
Let $\cM$ be a smooth manifold, $\mg$ a continuous metric, 
and $\mu_\mg$ the induced volume measure.
Then, 
\begin{enumerate}[(i)]
\item a set $M \subset \cM$ is measurable if and only 
if $M$ is $\mu_\mg$-measurable,
\item a function $f: \cM \to \C$ is measurable if and only if
	 it is $\mu_\mg$-measurable,
\item a set $Z \subset \cM$ is a null set if and only if
$\mu_\mg(Z) = 0$, 
\item a property $P$ holds a.e. in $\cM$
if and only if it holds $\mu_\mg$-a.e.
\end{enumerate} 
\end{proposition}
\begin{proof}
Take a covering of $\cM$ by coordinate charts, each of
which have compact closure. 
 Inside each chart, we can apply Lemma \ref{Lem:WeightedMeas}
with $f = \sqrt{\det \mg}$. The conclusion then follows
immediately.
\end{proof} 

Let us now fix a continuous metric $\mg$ on our
smooth manifold $\cM$.
The $\Lp{p}(\Tensors[r,s]\cM,\mg)$ for $1 \leq p < \infty$ is defined
as the space of sections $\xi \in \Sect(\Tensors[r,s]\cM)$ such that
$$ \int_{\cM} \modulus{\xi(x)}_{\mg(x)}^p\ d\mu(x) < \infty.$$
The $\Lp{p}$ norm is then simply the quantity
$$\norm{\xi}_p = \cbrac{\int_{\cM} \modulus{\xi(x)}_{\mg(x)}^p\ d\mu(x)}^{\frac{1}{p}}.$$
Similarly, $\Lp{\infty}(\Tensors[r,s]\cM,\mg)$ consist of 
sections $\xi \in \Sect(\Tensors[r,s]\cM)$ such that there
exists $C > 0$ satisfying $\modulus{\xi(x)}_\mg(x) \leq C$ for
$x$-a.e. Then,
$$\norm{\xi}_\infty = \inf\set{C: \modulus{\xi(x)}_{\mg(x)} \leq C\ x\text{-a.e.}}.$$

By $\Sob{1,p}(\cM,\mg)$, we denote the $\Lp{p}$ Sobolev space over
functions. That is, we let $S_p = \set{u \in \Ck{\infty}\intersect \Lp{p}: \conn u \in \Lp{p}}$
and define $\Sob{1,p}(\cM,\mg)$ as the closure of $S_p$ with respect to the
Sobolev norm 
$$\norm{u}_{\Sob{1,p}} = \norm{u}_p + \norm{\conn u}_p.$$
The space $\Sob[0]{1,p}(\cM, \mg)$ is defined by
closing $\Ck[c]{\infty}(\cM)$ under the same norm. Note that
in general, $\Sob[0]{1,p}(\cM,\mg) \neq \Sob{1,p}(\cM,\mg)$.
However, the following is true. 

\begin{proposition}
\label{Prop:OpSob}
The operator $\conn[p]:\Ck{\infty} \intersect \Lp{p} \to \Ck{\infty} \intersect \Lp{p}$
is closable in $\Lp{p}$, as is the operator $\conn[c]: \Ck[c]{\infty}(\cM) \to \Ck[c]{\infty}(\cM)$.
Moreover, $\Sob{1,p}(\cM,\mg) = \dom(\close{\conn[p]})$ and $\Sob[0]{1,p}(\cM,\mg)=\dom(\close{\conn[c]})$.
\end{proposition}
\begin{proof}
Since $\conn[c] \subset \conn[p]$, it is enough to 
show that $\conn[p]$ is closable. 
For this, all we need to show is that whenever
$u_n \in S_p$ with $u_n \to 0$ and $\conn[p] u_n \to v$, 
then $v = 0$.

Fix a chart $(U, \psi)$ such that $\close{U}$ is compact 
and note that
$\conn[p] u_n \to v$ implies that
$\conn[p] u_n \to v$ in $\Lp{p}(U,\mg)$.
Then,
\begin{multline*}
\int_{U} \modulus{\conn u_n - v}^p_\mg\ d\mu_\mg 
	= \int_{\psi(U)} \modulus{\pullb{\psi^{-1}}\conn u_n - \pullb{\psi^{-1}} v}\ \sqrt{\det \mg}\ d\Leb \\
	\geq \cbrac{\essinf_{x \in U} \sqrt{\det \mg(x)}} \int_{\psi(U)} \modulus{\conn \pullb{\psi^{-1}}u_n - \pullb{\psi^{-1}}v}\ d\Leb.
\end{multline*}
The $\essinf_{x \in U} \sqrt{\det \mg(x)} > 0$ since $\close{\psi(U)}$ is compact
and therefore, $\conn[p] u_n \to v$ in $\Lp{p}(U,\mg)$ implies
that $\conn[p] \pullb{\psi^{-1}}u_n \to \pullb{\psi^{-1}}u$ in $\Lp{p}(\psi(U), \Leb)$. 
But $\conn$ here is the 
exterior derivative $\extd$ which commutes with pullbacks, and hence,
since $\extd$ is closable in $\R^n$,
we have that $\pullb{\psi^{-1}} v = 0$ in $\phi(U)$.
Then, $v = 0$  in $U$ and since we can cover the manifold $\cM$ by countably 
many such charts $(U,\psi)$, we obtain that $v = 0$.
That the Sobolev spaces can be equated with the domains
of the closure of the relevant operator is then immediate.
\end{proof}

In the more specific case of a complete metric,
we obtain that the two Sobolev spaces are
equal.
\begin{proposition}
Suppose that $\mg$ is a continuous metric
that is complete. Then, $\Sob[0]{1,p}(\cM,\mg) = \Sob{1,p}(\cM,\mg)$.
\end{proposition}
\begin{proof}
Fix $\epsilon > 0$ and note that since $\mg$ is a continuous
and complete metric, there exists $R_\epsilon > 0$ such that
whenever $r \geq R_\epsilon$,
there exists a $\met_\mg$-Lipschitz
cutoff function $\phi_r:\cM \to [0,1]$ such that $\phi_r = 1$ on $B_r$, 
$\spt \phi_r$ is compact 
and $\modulus{\conn \phi_r}^p = (\Lip \psi_r)^p < \epsilon$
almost-everywhere. 
%We note that due to the low regularity of
%$\mg$, the compactness of $\spt \phi_r$
%follows from the compactness of $\close{B}_r$
%by invoking the Hopf-Rinow theorem for length-metric spaces.

Now, let $u \in \Sob{1,p}(\cM,\mg)$
and note that 
$$\int_{\cM} (\modulus{u}^p + \modulus{\conn[p] u}^p)\ d\mu_\mg 
	= \lim_{r \to \infty} \int_{B_r} (\modulus{u}^p + \modulus{\conn u}^p)\ d\mu_\mg.$$
As a consequence, note that we can find another $R'_\epsilon > 0$ large
enough so that whenever $r \geq R'_\epsilon$, 
$$ \int_{\cM\setminus B_r} \modulus{u}^p\ d\mu_\mg < \epsilon
\ \text{and}\ \int_{\cM\setminus B_r} \modulus{\conn[p] u}^p\ d\mu_\mg < \epsilon.$$

Also note that $\Sob{1,p}(\cM,\mg)$ remains invariant
under multiplication by bounded Lipschitz functions. Therefore, upon 
taking $R$ to be the larger of $R_\epsilon$ and $R'_\epsilon$,
we note that for almost-every $x \in \cM$
$$
\modulus{\conn[p] u - \conn[p](\phi_r u)}^p
	\lesssim \modulus{\conn[p] u - \phi_r \conn[p] u}^p + \modulus{u\conn \phi_r}^p,$$
and that
$$ \int_{B_r} \modulus{\conn[p] u - \phi_r \conn[p] u}^p\ d\mu_\mg = 0.$$
Furthermore,
$$
\int_{\cM \setminus B_r} \modulus{(1 - \phi_r) \conn[p]u}^p\ d\mu_\mg
	\leq \int_{\cM\setminus B_r} \modulus{\conn[p] u}^p\ d\mu_\mg
	< \epsilon,$$
and
$$\int_\cM \modulus{u \conn[p]\phi_r}^p\ d\mu_\mg 
	= \int_{\cM\setminus B_r} \modulus{u \conn[p]\phi_r}^p\ d\mu_\mg
	<  \epsilon \int_{\cM\setminus B_r} \modulus{u}^p\ d\mu_\mg
	< \epsilon^2.$$
Therefore, 
$$\int_{\cM} \modulus{\conn[p] u - \conn[p](\phi_r u)}^p\ d\mu_\mg < \epsilon +\epsilon^2.$$
Thus, we can approximate $u \in \Sob{1,p}(\cM,\mg)$ by a sequence
$u_n \in \Sob{1,p}(\cM,\mg)$ such that $\spt u_n$ is compact. 

By what we have just proved, to establish our claim, 
it suffices to restrict our attention to $u \in \Sob{1,2}(\cM,\mg)$
with $\spt u$ compact.
Let us take a partition of unity $\set{\xi_i}$ with respect to the compact
charts $(U_i,\psi_i)$ covering $\spt u$. Then,
$u = \sum_{i=1}^N \xi_i u$ and on denoting the
standard symmetric mollifier in $\R^n$ by $\eta^\epsilon$, define 
$$u_\epsilon = \sum_{i=1}^N \pullb{\psi_i}(\eta^\epsilon \convolve \pullb{\psi^{-1}_i} \xi_i u),$$
for $\epsilon > 0$ small enough so that 
$\spt (\eta^\epsilon \convolve \pullb{\psi^{-1}_i} \xi_i u) \subset \xi_i(U_i)$.
Now, $u_\epsilon \in \Ck[c]{\infty}(\cM)$ and
\begin{align*}
\int_{\cM} \modulus{u - u_\epsilon}^p\ d\mu_\mg 
	&\leq \sum_{i=1}^N \int_{U_i} \modulus{\pullb{\psi_i}(\eta^\epsilon \convolve \pullb{\psi^{-1}_i} \xi_i u) 
		- \xi_i u}^p\ d\mu_\mg \\
	&= \sum_{i=1}^N \int_{\psi_i(U_i)} \modulus{\eta^\epsilon \convolve \pullb{\psi^{-1}_i} \xi_i u 
		- \pullb{\psi^{-1}_i}\xi_i u}^p\ \sqrt{\det\mg}\ d\Leb. 
\end{align*}
But $\mg$ is continuous and $\close{\psi_i(U_i)}$ is compact, and therefore, 
there exist $C_1, C_2 > 0$ such that
$C_1 \leq \sqrt{\det \mg} \leq C_2$. 
Therefore, $\eta^\epsilon \convolve \pullb{\psi^{-1}_i} \xi_i u \to \pullb{\psi^{-1}_i} \xi_i u$
in $\Lp{p}(\R^n)$ and hence, $u_\epsilon \to u$
in $\Lp{p}(\cM)$. Since $\conn[p] = \conn = \extd$ on functions,
and since the exterior derivative commutes with pullbacks,
we obtain by a similar argument that $\conn[p]u_\epsilon \to \conn[p]u$.
\end{proof}

Let $\Forms[k](\cM)$ denote the bundle of $k$-forms, 
and $\Forms(\cM)$  the bundle of forms.
Recall that the \emph{exterior derivative} is then
the differential operator $\extd:\Ck{\infty}(\Forms[k](\cM)) \to 
\Ck{\infty}(\Forms[k+1](\cM))$. 
For $k=1$, this is just $\conn$.
Let 
$$S^\extd_p = \set{u \in \Ck{\infty} \intersect 
\Lp{p}(\Forms(\cM)): \extd u \in \Lp{p}(\Forms(\cM))}$$
and define the norm 
$\norm{u}_{\extd,p} = \norm{u}_p + \norm{\extd u}_p.$
We define the \emph{Sobolev spaces of \extd} as follows.
Let the space $\Sob{\extd,p}(\cM,\mg)$ 
denote the closure of $S^\extd_p$
under $\norm{\mdot}_{\extd,p}$, and $\Sob[0]{\extd,p}(\cM,\mg)$
as the closure of $\Ck[c]{\infty}(\Forms(\cM))$ under
the same norm. We refrain from proving
a result similar to Proposition \ref{Prop:OpSob}
here since we will prove it later
with the aid of better tools.
 
While we have defined these function spaces for general 
$p$, we shall only concentrate here on the case of $p = 2$. 
Note that $\Lp{2}(\Tensors[r,s]\cM,\mg)$ is a Hilbert space, 
with the induced inner product given by
$$ \inprod{u,v} = \int_{\cM} \inprod{u(x),v(x)}_{\mg(x)}\ d\mu_\mg(x).$$ 
It is a standard fact that $\Ck[c]{\infty}(\Tensors[r,s]\cM)$
is a dense subset of $\Lp{p}(\Tensors[r,s]\cM,\mg)$ for $1 \leq p < \infty$, so 
this is in particular true for $\Lp{2}(\cM,\mg)$. Therefore, operator theory 
guarantees the existence of closed, densely-defined
adjoints to $\close{\conn[2]}$ and $\close{\conn[c]}$
which we identify as the \emph{divergence} operators
$\divv_{\mg} = -\adj{\close{\conn[2]}}$ and
$\divv_{0,\mg} = -\adj{\close{\conn[c]}}$.

\subsection{Axelsson-Keith-McIntosh framework in Hilbert spaces}
\label{Sect:AKM}

At the core of our achievements in this paper
lies the \emph{Axelsson-Keith-McIntosh} framework.
First formulated in 2005 in their paper \cite{AKMc}, it provides
a \emph{first-order} framework unifying the resolution of
many problems in modern harmonic analysis, including the 
boundedness of the Cauchy integral operator on Lipschitz curves,
the Kato square root problem for systems on $\R^n$, 
and the Kato square root problem on compact manifolds.

As aforementioned in \S\ref{Sect:Intro}, 
this framework has been used and adapted 
many times to solve a range of different problems.
The framework at the level of Hilbert spaces as
presented in \cite{AKMc} is a Dirac-type setup
 consisting of three hypotheses
(H1)-(H3) as follows. We emphasise here
that $\Hil$ is a Hilbert space and $\Gamma$ an unbounded operator
on $\Hil$. 
 
\begin{enumerate}[(H1)]
\item The operator $\Gamma: \dom(\Gamma) \subset \Hil \to \Hil$ 
	is a closed, densely-defined and nilpotent operator, 
	by which we mean $\ran(\Gamma) \subset \nul(\Gamma)$,

\item $B_1, B_2 \in \bddlf(\Hil)$ 
	and there exist $\kappa_1, \kappa_2 > 0$ 
	satisfying the accretivity conditions 
	$$\re\inprod{B_1 u, u} \geq \kappa_1 \norm{u}^2
	\ \text{and}\ 
	\re\inprod{B_2 v, v} \geq \kappa_2 \norm{v}^2,$$
	for $u \in \ran(\adj{\Gamma})$ and $v \in \ran(\Gamma)$, and

\item $B_1 B_2 \ran(\Gamma) \subset \nul(\Gamma)$ and 
	$B_2 B_1 \ran(\adj{\Gamma}) \subset \nul(\adj{\Gamma})$. 
\end{enumerate}

Let us now define $\Pi_B = \Gamma + B_1 \adj{\Gamma}B_2$
with domain $\dom(\Pi_B) = \dom(\Gamma) \intersect \dom(B_1\adj{\Gamma}B_2)$.
The conditions (H1)-(H3) guarantee
that $\Pi_B$ is a \emph{bi-sectorial} operator. In particular, this means
that $Q_t^B = t\Pi_B(\iden + t^2 \Pi_B^2)^{-1}$ is a
bounded operator. For a more complete treatment 
of sectorial and bi-sectorial
operators and their functional calculi, 
we recommend \cite{ADMc} by Albrecht, Duong, and McIntosh, 
\cite{Morris2} by Morris or the book \cite{Haase}
by Haase. 

To say that $\Pi_B$ satisfies \emph{quadratic estimates}
means that
\begin{equation*}
\label{Qest}
\tag{Q}
\int_{0}^\infty \norm{Q_t^B u}^2\ \frac{dt}{t} \simeq \norm{u}^2, 
\end{equation*}
for all $u \in \close{\ran(\Pi_B)}$.
Such estimates are the fundamental objects
of study, for they provide access to the tools
of harmonic analysis in proving Kato square root
theorems. The following is the main consequence.
Note that here, the function $\chi_{+}$
takes value $1$ on the right half of the complex
plane and $\chi_{-}$ takes $1$ on the left half.

\begin{theorem}[Kato square root type theorem]
\label{Thm:KatoType}
Suppose that $\Pi_B$ satisfies the
quadratic estimates \eqref{Qest}. Then,
\begin{enumerate}[(i)]
\item there is a spectral decomposition: 
	$$\Hil = \nul(\Pi_B) \oplus E_+ \oplus E_-$$
	where $E_{\pm} = \chi_{\pm}(\Pi_B)$ (the sum is, in general, non-orthogonal), 
\item $\dom = \dom(\Pi_B) = \dom(\Gamma) \intersect \dom(B_1\adj{\Gamma} B_2) = \dom(\sqrt{\Pi_B^2})$
	with 
	$$\norm{\Pi_B u} \simeq \norm{\Gamma u} + \norm{\adj{\Gamma} B_2 u} \simeq \norm{\sqrt{\Pi_B^2 u}}$$
	for all $u \in \dom$.
\end{enumerate}
\end{theorem}

This theorem and  a description of
its consequences (in particular the stability of
the perturbation $\tilde{B} \mapsto \Pi_{B + \tilde{B}}$
as Lipschitz estimates) can be found in  \cite{AKMc}. 
A slightly more general exposition of similar results 
are contained in \S1.8 of \cite{BThesis}.

Let us conclude this section by demonstrating the connection
between this setup and the aforementioned theorem 
to the Kato square root problem (for functions). Let $(\cM, \mg)$ be a manifold and 
$\Hil = \Lp{2}(\cM) \oplus \Lp{2}(\cM) \oplus \Lp{2}(\cotanb\cM)$.
Let $S = (\iden, \conn): \Lp{2}(\cM) \to \Lp{2}(\cM) \oplus \Lp{2}(\cM)$,
and suppose 
$A \in \Lp{\infty}(\bddlf(\Lp{2}(\cM) \oplus \Lp{2}(\cotanb\cM)))$ 
and $a \in \Lp{\infty}(\bddlf(\Lp{2}(\cM)))$ satisfying
$\re \inprod{au,u} \geq \kappa_1 \norm{u}^2$
and $\re \inprod{ASv, v} \geq \kappa_2 \norm{v}_{\SobH{1}}^2$ 
for all $u \in \Lp{2}(\cM)$ and $v \in \SobH{1}(\cM)$.
Then, set 
$$
\Gamma =\begin{pmatrix} 0 & 0 \\ S & 0 \end{pmatrix},\ 
B_1 = \begin{pmatrix} a & 0 \\ 0 & 0 \end{pmatrix},\ 
B_2 = \begin{pmatrix} 0 & 0 \\ 0 & A \end{pmatrix}.$$

If the quadratic estimates \eqref{Qest} hold, 
then the conclusions of Theorem \ref{Thm:KatoType} 
are true, and in particular an easy calculation 
will reveal that the following Kato square root estimate holds:
\begin{align*}
\label{Kato}
\tag{K}
\dom(\sqrt{a\adj{S}AS}) = \SobH{1}(\cM)\ \text{and}\ 
\norm{\sqrt{a\adj{S}AS}u} \simeq \norm{u}_{\Sob{1}}
\end{align*}
for all $u \in \Sob{1}(\cM)$.
\section{Rough metrics and their properties}
\label{Sect:Rough}

In this section, we formulate a
notion of a \emph{rough metric}
as a Riemann-like metric which we
permit to be of low regularity and degenerate.
We will see in subsequent parts of this paper
that such metrics provide a sufficiently 
large class of geometries on which the 
quadratic estimates we consider
are stable. We establish some of their key properties, 
their associated Lebesgue spaces, 
and consider implications for rough
geometries when there are regular geometries
nearby.
  
\subsection{Rough metrics}
\label{Sect:RoughM}

We begin our discussion with the following definition,
noting that $\delta$ denotes
the usual Euclidean inner product.

\begin{definition}[Rough metric]
\label{Def:RoughMet} 
Suppose that $\mg \in \Sect(\Tensors[2,0]\cM)$
is symmetric and satisfies the
following \emph{local comparability condition}: for every $x \in \cM$, 
there exists
a chart $(U,\psi)$ near $x$ and constant $C \geq 1$ 
such that 
$$C^{-1} \modulus{u}_{\pullb{\psi}\delta (y)} 
	\leq \modulus{u}_{\mg(y)} \leq C \modulus{u}_{\pullb{\psi}\delta(y)}$$
for $u \in \tanb_y U$ and for almost-every $y \in U$.
Then we say that $\mg$ is a \emph{rough metric}. 
\end{definition}

By covering $\cM$ by a countable collection of
charts satisfying the local comparability 
condition, it is easy to see
that 
$0 < \modulus{u}_{\mg(x)} < \infty$ 
for $0 \neq u \in \tanb_x\cM$ for almost-every
$x \in \cM$.
As a consequence, we define the \emph{singular set} of $\mg$ as
$$\Sing(\mg) = \set{x \in \cM: \modulus{u}_{g(x)} = \infty\ 
	\text{or}\ \modulus{v}_{g(x)} = 0,\ v \neq 0}.$$
It is easy to see that
this set is of null measure.
We define the \emph{regular set} of $\mg$ by
$\Reg(\mg) = \cM \setminus \Sing(\mg)$.

The measurability of $\mg$, which if
we recall is exactly the measurability of its
coefficients inside coordinate charts
against the $\Leb$-measure, provides us
with a means to define a volume measure $\mu_\mg$.
First, we note the following.

\begin{proposition}
Let $(U,\psi)$ and $(V,\phi)$ be two coordinate
charts satisfying the local comparability condition 
such that $W = U \intersect V \neq \emptyset$.
Then, for any measurable set $A \subset W$,
$$
\int_{\psi(A)} \sqrt{\det \mg(x)}\ d\Leb(x) = 
	\int_{\phi(A)} \sqrt{\det \mg(y)}\ d\Leb(y).$$
\end{proposition}
\begin{proof}
Let $\set{x^i}$ and $\set{y^j}$ 
be the the coordinates in $(U, \psi)$ and $(V,\phi)$
respectively. Then, let us denote 
$\mg$ in $U$ and $V$ respectively as
$$\mg (x) = \mg_{ij}(\psi(x))\ dx^i \tensor dx^j
\ \text{and}\ 
\mg(y) = \mgt_{kl}(\phi(y))\ dy^k \tensor dy^l,$$
for almost all $x \in U$ and almost all $y \in V$.
Then, on the intersection, the transformation formula
is similar to the case of a continuous metric but
holding for almost every $x\in\cM$,
$$
\mg_{ij}(\psi(x)) = \mgt_{kl} (\phi \comp \psi^{-1}(x)) \partt[x^i]{y^k} \partt[x^j]{y^l}, 
$$
where $(\partt[x^i]{y^j}) = \Dir(\phi \comp \psi^{-1})$.
Letting $G_x = (\mg_{ij})$ and $\tilde{G}_{y} = (\mgt_{ij})$, 
we note that
$G_x = \Dir(\phi \comp \psi^{-1}) \tilde{G}_y \tp{\Dir(\phi \comp \psi^{-1})},$
and hence
$\det G_x = (\det G_y) (\det\Dir(\phi \comp \psi^{-1}))^2.$
Recall that for an integrable function $\xi: \psi(W) \to \R$, 
the change of variable formula is given by 
$$ 
\int_{\psi(A)} \xi(x)\ d\Leb(x) 
	= \int_{\phi(A)} (\xi \comp \psi \comp \phi^{-1})(y)
		\modulus{\det\Dir(\psi \comp \phi^{-1}(y))}\ d\Leb(y).$$ 
On noting that 
$(\psi\comp \phi^{-1})^{-1} = \phi \comp \psi^{-1}$,
and combining this with  our previous calculation of $\det G_x$,
the conclusion follows.
\end{proof}

As for the case of a continuous metric, we define
the induced measure $\mu_\mg$ by 
pasting together the local expression 
$$d\mu_\mg(x) = \sqrt{\det \mg(x)}\ d\Leb(x)$$
 via a smooth partition
of unity subordinate to a countable covering
of $\cM$ by charts satisfying the local comparability 
condition. The previous proposition guarantees
that this object is well-defined under a change of 
coordinates.  We
observe the following properties of this measure.

\begin{proposition}[Compatibility of the measure]
\label{Prop:CompMeas}
Let $\mg$ be a rough metric. Then, 
\begin{enumerate}[(i)] 
\item a set $A \subset \cM$ is measurable
if and only if it is $\mu_\mg$-measurable, 
\item a function $f:\cM \to \C$ is 
measurable if and only if it is $\mu_\mg$-measurable,
\item $Z$ is a set of null measure if and only if $\mu_\mg(Z) = 0$,
\item a property $P$ holds a.e. in $\cM$ if and only
	if it holds $\mu_\mg$-a.e., 
\end{enumerate}
\end{proposition}
\begin{proof}
Let $(U, \psi)$ be a chart
satisfying the local comparability condition. 
Then, inside each such chart, we can apply Lemma \ref{Lem:WeightedMeas}
since $0 < \modulus{u}_{\mg(x)} < \infty$ 
for almost-every $x \in \cM$ and $0 \neq u \in \tanb_x\cM$.
Thus, we are able to conclude the properties hold inside $U$
and  again by the observation that the manifold can be
covered by countably many such charts, the conclusion follows.
\end{proof}

We remark that if we had simply asked
for $\Sing(\mg)$ to be a null
measure set in place of the local comparability 
condition,
the conclusions of this proposition would still be valid.
In what is to follow, we will see that 
the local comparability condition becomes
important to establish regularity properties of the 
measure, as well as some desirable properties 
of Lebesgue and Sobolev spaces for such metrics.
First,  let us note the following lemma.

\begin{lemma}
\label{Lem:Det}
Let $B$ be a symmetric, positive matrix
on $\R^n$ such that  there exists $C \geq 1$
satisfying 
$C^{-1} \modulus{u} \leq \tp{u}Bu \leq C\modulus{u}.$
Then, $C^{-n} \leq \det B \leq C^n.$
\end{lemma}
\begin{proof}
Let $\inprod{\mdot, \mdot}$ be the standard
Euclidean inner product. Then, 
$\tp{u}Bu = \inprod{u, Bu} = \inprod{Bu, u}$
by the symmetry of $B$. Choosing $\modulus{u} = 1$,
we have that the numerical range
$\nr(B) \subset [C^{-1}, C]$. But 
$\spec(B) \subset \close{\nr(B)} \subset [C^{-1}, C]$.
Letting $\spec(B) = \set{b_1, \dots, b_n}$,
we have that 
$C^{-n} \leq \det B = \prod_{i=1}^n b_i \leq C^n$. 
\end{proof} 

Using this lemma, we are first able to prove the
following regularity result for the measure $\mu_\mg$.

\begin{proposition}
\label{Prop:BorelCpct}
The measure $\mu_\mg$ is Borel and finite on compact
sets.
\end{proposition}
\begin{proof}
That $\mu_\mg$ is Borel is a simple consequence of (i) 
in Proposition \ref{Prop:CompMeas}.

We show that it attains finite measure on compact sets.
Let $K \subset \cM$ be compact and let $\set{U_i}_{i=1}^N$
be a finite cover of $K$ by charts $(U_i, \psi_i)$ 
each of which satisfy the local comparability condition
with constants $C_i$. Inside $U_i$, we have that
$$\mu_\mg(U_i \intersect K) =  
	\int_{\psi_i(U_i \intersect K)} \sqrt{\det \mg(x)}\ d\Leb(x)
	\leq C_i^{\frac{n}{2}} \int_{\psi_i(U_i \intersect K)}\ d\Leb(x)
	< \infty.$$ 
Thus, $\mu_\mg(K) \leq \sum_{i=1}^N \mu_\mg(U_i \intersect K) < \infty.$
\end{proof}
\begin{remark} 
We do not know whether such a measure 
is also Borel-regular.
\end{remark}

We remark that we do not attempt to 
explore any distance metric properties
that rough metrics may induce, simply 
because such metrics may not even 
provide us with a sufficiently finite
length functional.
Although additional conditions can remedy such an 
effect, we are primarily concerned with
constructing Lebesgue and Sobolev spaces
with certain desirable properties,
and we will see in later parts that
a possible lack of a metric
structure will not hinder our efforts.

It is easy to see that every continuous Riemannian metric
is a rough metric. Thus, from this point
onwards, we shall simply refer to rough metrics
as metrics. In order to illustrate that our definition 
is non-trivial, we provide the following two examples of
degenerate rough metrics.

\begin{example}
\label{Ex:Ex1}
Let $\cM = \R^2$ and define $\mg(u,v) = \inprod{u,v}_{\R^2}$
on $\R^2 \setminus (\R \times 0)$ and $\mg(u,v) = 0$ on $\R \times 0$. 
Then, $\mg$ is a rough metric with $\mu_\mg = \Leb$.
This space can be seen as
$\R^2$ with the line $\R \times 0$ collapsed to a single point.
\end{example}

\begin{example}
\label{Ex:Ex2}
Let $(\cM, \mh)$ be a smooth manifold with a continuous
metric and let $Z \subset \cM$ be a null set.
Then, set $\mg = \mh$ on $\cM \setminus Z$ and
$\mg = 0$ on $Z$ and we have that $\mu_\mg = \mu_\mh$.
We remark that $Z$ can even be a dense subset.
\end{example}

\subsection{Lebesgue and Sobolev spaces}

In this section, we define Lebesgue 
and Sobolev spaces for rough metrics.
Since we are interested in their 
relationship to differential operators, 
we prove that compactly 
supported functions are dense in the $\Lp{p}$
theory, which, as in the continuous metric
case, allows us to obtain the Sobolev spaces
as domains of operators.

First, for a rough metric $\mg$,
we define the space $\Lp{p}(\Tensors[r,s]\cM,\mg)$
precisely as we did for a continuous metric. 
Note that such spaces can be defined
for a measure space alone. 
The Sobolev spaces $\SobH{1}(\cM,\mg)$ and $\SobH[0]{1}(\cM,\mg)$ 
are also defined similarly. 
 
While for continuous metrics
it is clear that 
$\Ck[c]{\infty}(\Tensors[r,s]\cM)$ is dense
in  $\Lp{p}(\Tensors[r,s]\cM)$, this
is not immediate for rough metrics.
We dedicate some energy to verify this is indeed
the case. 
To aid us, we first demonstrate the following simple 
lemma.
 
\begin{lemma}
For every smooth manifold $\cM$, there exist a
sequence of open sets $U_i$ such that $\close{U_i}$
is compact, $U_i \subset U_{i + 1}$, 
and $\cM = \union_i U_i = \lim_i U_i$.
\end{lemma}
\begin{proof}
Fix $x \in \cM$ and let $(U, \psi)$ be a
chart near $x$. Then, let $B(x,r) \subset \psi(U)$
be a Euclidean ball, and let $V_x = \psi^{-1}(B(x,1/2r))$.
It is easy to see that $V_x$ is open and $\close{V_x}$
is compact. Since $\cM$ is second-countable, 
we are able to extract a countable subcollection
$\set{V_i}$. Now, set $U_i = \union_{j=1}^{i} V_i$
which finishes the proof.
\end{proof}

An immediate consequence is the following approximation result.
\begin{lemma}
\label{Lem:CptApprox}
For $p \in [1,\infty)$ and for every $u \in \Lp{p}(\Tensors[r,s]\cM,\mg)$,
there exists a sequence $u_n \in \Lp{p}(\Tensors[r,s]\cM,\mg)$
such that $\spt u_n$ is compact
and $u_n \to u$ in $\Lp{p}$. 
\end{lemma}
\begin{proof}
Let $u \in \Lp{p}(\Tensors[r,s]\cM,\mg)$
and fix $\epsilon > 0$. Let $U_i$ be the collection
of sets guaranteed by the previous lemma.
We claim that there exists an $N$ such that 
$$ \int_{\cM\setminus U_N} \modulus{u}\ d\mu_\mg < \epsilon.$$
To argue by contradiction, suppose not. That is,
for every $i$, 
$$\int_{\cM \setminus U_i} \modulus{u}^p\ d\mu_\mg \geq \epsilon.$$
But we have that 
$$ \int_{\cM} \modulus{u}^p\ d\mu_\mg = 
	\int_{U_i} \modulus{u}^p\ d\mu_\mg + \int_{\cM\setminus U_i} \modulus{u}^p\ d\mu_\mg
	\geq \int_{U_i} \modulus{u}^p\ d\mu_\mg + \epsilon.$$
Also, $\lim_{i} U_i = \cM$ and therefore, we have that 
$$\int_{\cM} \modulus{u}^p\ d\mu_\mg \geq \int_{\cM} \modulus{u}^p\ d\mu_\mg + \epsilon,$$
which means that $\epsilon \leq 0$, which is a contradiction.
Since each $\close{U_i}$ is compact, the sequence $u_n$
can be obtained simply by setting $\epsilon = 1/n$
to extract sets $U'_{n}$
and setting $u_n = \ind{U'_n} u$.
\end{proof} 

Note that we did not use any properties
of the metric $\mg$ in proving this lemma.
However, in the proof of the following proposition,
the locally comparability condition
becomes of crucial importance.

\begin{proposition}
\label{Prop:SCdense}
Whenever $\mg$ is a rough metric and $p \in [1,\infty)$,
the space $\Ck[c]{\infty}(\Tensors[r,s]\cM)$ is dense in 
$\Lp{p}(\Tensors[r,s]\cM,\mg)$.
\end{proposition}
\begin{proof}
Fix $u \in \Lp{p}(\Tensors[r,s]\cM,\mg)$.
By Lemma \ref{Lem:CptApprox}, we can assume that
$\spt u$ is compact, and even further, without the
loss of generality, let us assume that $\spt u \subset U$,
where $U$ is a chart where the local comparability condition 
is valid.

Noting that $u = u_I^J\ dx^I \tensor \partial_{x^J}$
(where $I$ and $J$ are multi-indices), 
write $u_\epsilon = \pullb{\psi^{-1}}(\eta^\epsilon \convolve u_I^J) dx^I  \tensor \partial_{x^J}$.
Then, for some $\kappa > 0$, we have that $\spt u_\epsilon \subset U$ 
for all $\epsilon < \kappa$. It is easy to see that $u_\epsilon \in \Ck[c]{\infty}(\Tensors[r,s]\cM)$.
By invoking Lemma \ref{Lem:Det}
$$
\int_{U} \modulus{u_\epsilon - u}_{\mg}^p\ d\mu
	\lesssim \int_{U} \modulus{u_\epsilon - u}_{\pullb{\psi}\delta}^p\ d\mu.$$
But, inside $(U,\psi)$, $d\mu(x) = \sqrt{\det \mg(x)}\ d\Leb(x)$ and 
by Lemma \ref{Lem:Det}, on writing $\mg(u,v) = \tp{u}Gv$ at regular points,
we obtain that $C^{-\frac{n}{2}} \leq \sqrt{\det\mg} \leq C^{\frac{n}{2}}$.
Thus, 
$$\int_{U} \modulus{u_\epsilon - u}_{\mg}\ d\mu
	\lesssim \int_{\psi(U)} \modulus{\pullb{\psi}u_\epsilon - \pullb{\psi} u}_{\delta} \ d\Leb
	\to 0$$ 
by the standard results on mollification in Euclidean space.
\end{proof}

By using this proposition, we are able to 
assert the following important properties of
$\conn[p]$ and obtain the Sobolev spaces
associated to $\mg$ as domains of the closure of this operator.

\begin{proposition}
\label{Prop:OpSobRough}
For a rough metric $\mg$,
$\conn[p]: \Ck{\infty}\intersect \Lp{p}(\cM) \to \Ck{\infty}\intersect \Lp{p}(\cotanb\cM)$
and $\conn[c]: \Ck[c]{\infty}(\cM) \to \Ck[c]{\infty}(\cotanb\cM)$ are
closable, densely-defined operators. Furthermore,
$\Sob{1,p}(\cM) = \dom(\close{\conn[p]})$ and $\Sob[0]{1,p} = \dom(\close{\conn[c]})$.
\end{proposition}
\begin{proof}
Since $\conn[c] \subset \conn[p]$ it suffices
to prove the statement for $\conn[p]$. The fact that 
$\conn[p]$ and $\conn[c]$ are densely-defined is immediate
from Proposition \ref{Prop:SCdense}.

To show that $\conn[p]$ is closable,
let $u_n \in \Ck[c]{\infty} \intersect \Lp{p}(\cM)$
such that $\conn[p] u_n \in \Lp{p}(\cM)$ and $u_n \to 0$
and $\conn[p] u_n \to v$. It suffices to show
that $v = 0$ a.e. in a countable collection of 
charts $U_i$ satisfying the local comparability 
condition. For this, we can replicate the proof 
of Proposition \ref{Prop:OpSob}
since we only require that the quantity
$
\essinf_{x \in U} \sqrt{\det \mg(x)} > 0,$
which we have as a consequence of the local comparability 
condition coupled with Lemma \ref{Lem:Det}.
Obtaining the Sobolev spaces $\Sob{1,p}(\cM)$
and $\Sob[0]{1,p}(\cM)$ as
$\dom(\close{\conn[p]})$ and $\dom(\close{\conn[c]})$ 
is then immediate.
\end{proof}

As a consequence of Proposition
\ref{Prop:OpSobRough}, in the $\Lp{2}$ theory,
we define the divergence
operators to be $\divv_\mg = -\adj{\conn_2}$
and $\divv_{0,\mg} = -\adj{\conn_0}$.
Note that since $\close{\conn_2}$
and $\close{\conn_0}$ are closed
and densely-defined, so are $\divv_\mg$ and $\divv_{0,\mg}$
by the reflexivity of $\Lp{2}$.

\subsection{Uniformly close metrics and their properties}

Since we are ultimately interested in demonstrating
how to pass quadratic estimates between two 
geometries that are close, we define an appropriate
notion of closeness. We also demonstrate
how certain desirable properties of one metric are preserved for
the other nearby geometry.
Our starting point is the following definition.

\begin{definition}[Uniformly close metrics]
Let $\mg$ and $\mgt$ be two rough metrics and
suppose there exists $C \geq 1$ such that
$$ C^{-1} \modulus{u}_{\mgt(x)} \leq \modulus{u}_{\mg(x)} 
	\leq C \modulus{u}_{\mgt(x)},$$
for $u \in \tanb_x \cM$  and almost-every $x$ in $\cM$.
Then, we say that $\mg$ and $\mgt$ are 
\emph{uniformly close} or \emph{$C$-close}.
If the inequality holds everywhere,
then we say that the two metrics are
\emph{$C$-close everywhere}.
\end{definition}

By Proposition \ref{Prop:CompMeas}, this particularly 
means that the two uniformly close metrics $\mg$ and $\mgt$ 
satisfies the
inequality in the definition $\mu_\mg$-a.e as well as  
$\mu_\mgt$-a.e. Also, note that this
 inequality is invariant under interchanging
$\mg$ and $\mgt$.
Heuristically speaking, 
this condition captures that these
two geometries look
almost-everywhere uniformly close as viewed from either metric,
and that one geometry is almost-everywhere 
trapped by a uniform scaling of
the other.
 
%Let us quickly demonstrate a property which may not be
%preserved.  For example, suppose
%that $\mgt$ is a continuous, complete metric.
%This condition does not actually preclude
%the metric $\mg$ from being incomplete, and
%as we mentioned before, the topology obtained
%from $\mg$ will be coarser than the
%topology of $\cM$. In contrast, the topology 
%induced by $\mgt$ will be equivalent to that of
%$\cM$. This shows that at least some properties
%will not be preserved and gives some credence
%to the non-triviality of this notion of closeness.

A first result we prove is that 
for the continuous case, the
notion of $C$-close and $C$-close
everywhere are equivalent. 
For the purpose of convenience, 
from this point onwards, 
let us denote the largest set on which
the $C$-close inequality holds
by $\sR$.

\begin{proposition}
\label{Prop:ContClose}
Two continuous metrics $\mg$ and $\mgt$
are  $C$-close if and only if they 
are $C$-close everywhere. 
\end{proposition}
\begin{proof}
The direction $C$-close everywhere
implies $C$-close is easy. So we shall concentrate
on the mildly harder opposite direction.

If we prove that $\cM = \sR$, 
then we are done. To draw a contradiction,
suppose not and fix $x \in \cM \setminus \sR$. 
First, we claim that there exists a sequence 
$x_n \in \sR$ 
such that $x_n \to x$. We argue this by contraction, 
so suppose this is not true.
Then, that means that there exists some open set $U_x$ near
$x$ such that $U_x \intersect \sR = \emptyset$.
The only way this could happen is if $U_x \subset \cM\setminus\sR$.
Thus, $U_x$ must be a null measure set. However, in some coordinate
chart $(V,\psi)$ near $x$, there exists some $r > 0$ such that
$U_x' = \psi^{-1}(B_r(\psi(x))) \subset U_x$. This implies
that $U_x'$ is a set of measure zero. However,
$\psi(U_x') = B_r(\psi(x))$ which does not 
attain zero Lebesgue measure, and hence, we 
arrive at a contradiction.

Now, consider a chart $(U, \psi)$ near $x$,
and for $n \geq N$, we have that 
$x_n \in U$.
Furthermore, via the chart,
we can define a smooth $u':U \to \tanb\cM$
such that $u'(x) = u$
and $C^{-1} \modulus{u'(x_n)}_{\mgt(x_n)} \leq \modulus{u'(x_n)}_{\mg(x_n)} 
	\leq C \modulus{u'(x_n)}_{\mgt(x_n)}.$
Each quantity is in this inequality is continuous and therefore, 
we can interchange the function and the limit. Thus, 
we find that
$C^{-1} \modulus{u}_{\mgt(x)} \leq \modulus{u}_{\mg(x)}
	\leq C \modulus{u}_{\mgt(x)}.$
Therefore, $x \in \sR$ which is a contradiction.
\end{proof} 

In order to explore the deeper
properties of $C$-close metrics, 
it is convenient to be able
to write one metric in terms of the
other. To aid us in this direction, 
we first prove the following lemma.

\begin{lemma}
\label{Lem:IPRep}
Let $V$ be a vector space of dimension $n$
and let $\mh_1$ and $\mh_2$ be two
inner products on $V$. Then, there exists
a bounded, symmetric, positive operator 
$B: V \to V$ such that
$\mh_1(Bu,v) = \mh_2(u,v)$.
\end{lemma}
\begin{proof}
Since $V \cong \R^n$, we assume without loss of generality 
that  $V = \R^n$. Then, we can write 
$\mh_i(u,v) = \tp{u}H_iv$ with $H$ being a 
positive, symmetric matrix. Indeed,
such a matrix is diagonalisable,
and so let us write $H_i = \tp{P}_i D_i P_i$
where $P_i$ is the matrix of eigenvectors and $D_i$
is the corresponding diagonal. Again, by the 
properties of $P_i$ and $D_i$, 
$H_i = \tp{P}_i \tp{\sqrt{D_i}} (\sqrt{D_i} P_i) = \tp{(\sqrt{D}P_i)}(\sqrt{D}P_i)$.

First, we show there exists a matrix $A:\R^n \to \R^n$ 
such that $\mh_1(Au,Av) = \mh_2(u,v)$.
This is equivalent to asking 
$\tp{(Au)}\tp{(\sqrt{D_1}P_1)} (\sqrt{D_1}P_1)Av = \tp{u}\tp{(\sqrt{D_2}P_2)}(\sqrt{D_2}P_2)v.$
Thus, it suffices to solve for
$\sqrt{D_1}P_1 A = \sqrt{D_2}P_2$. But $D_i$ are invertible
by their positivity and therefore, $A = \tp{P}_1\sqrt{D_1^{-1}}\sqrt{D_2}P_2$.
Then it is easy to see that $B = \adj{A}A = \tp{P}_2 D_1^{-1} D_2P_2$.
\end{proof}

An immediate consequence of this is the following, 
which allows us to capture the difference between 
two metrics as a $(1,1)$-tensor field.

\begin{proposition}
\label{Prop:OpExist}
Let $\mg$ and $\mgt$ be two rough metrics that are 
$C$-close. Then, there exists $\B \in \Sect(\cotanb\cM \tensor \tanb\cM)$
such that it is symmetric,  almost-everywhere  positive and invertible, and
$$\mgt_x(\B(x)u,v) = \mg_x(u,v)$$
for almost-every $x  \in \cM$. Furthermore, 
for almost-every $x \in \cM$, 
$$C^{-2} \modulus{u}_{\mgt(x)} \leq \modulus{\B(x)u}_{\mgt(x)} \leq C^2 \modulus{u}_{\mgt(x)},$$
and the same inequality with $\mgt$ and $\mg$ interchanged.
If $\mgt \in \Ck{k}$ and $\mg \in \Ck{l}$ (with $k, l \geq 0$),
then the properties of $\B$ are valid for all $x \in \cM$ and
$\B \in \Ck{\min\set{k,l}}(\cotanb\cM \tensor \tanb\cM).$
\end{proposition}
\begin{proof}
As before, let $\sR$ be the largest set on which the $C$-close 
inequality holds. Then, for $x \in \sR$, 
we invoke Lemma \ref{Lem:IPRep} to obtain 
a $B_x$ such that 
$\mgt_x(B_xu,v) = \mg_x(u,v)$ 
for $u, v \in \tanb\cM$. Then, set $\B(x) = B_x$.
Near $x$, the coefficients of $\B$
consist of the coefficients of $\mgt$, $\mg$
and their inverses,
we have that $\B \in \Sect(\cotanb\cM \tensor \tanb\cM)$.
That it is almost-everywhere positive, symmetric
and invertible comes from the fact that $\cM \setminus \sR$
is a null measure set.

By noting that for such a point $x \in \sR$,
$\mgt_x(B_xu,v) = \mgt_x(\sqrt{B}_xu, \sqrt{B}_xv)$,
we have by the $C$-close condition
that 
$$C^{-1}\modulus{u}_{\mgt(x)}\leq \modulus{u}_{\mg(x)} = \modulus{\sqrt{B}_x u}_{\mgt} \leq C \modulus{u}_{\mgt(x)},$$
from which the inequality in the conclusion of the theorem follows.
The fact that this remains unchanged under
the interchange of $\mgt$ and $\mg$ is obvious.

If $\mgt \in \Ck{k}$ and $\mg \in \Ck{l}$, then
$\sR = \cM$ by Proposition \ref{Prop:ContClose} 
and so the conclusions we have obtained so far are valid 
everywhere.
As aforementioned, $\B$ consists of the coefficients of
$\mgt$, $\mg$ and their inverses near $x$, and so it follows
that $\B \in \Ck{\min\set{k,l}}(\cotanb\cM \tensor \tanb\cM)$.
\end{proof}
\begin{remark}
By denoting the canonical extensions of these
metrics to $\Tensors[r,s]\cM$ by the 
same symbols, we note that we can prove
there exists $\B \in \Sect(\Tensors[s,r]\cM \tensor \Tensors[r,s]\cM)$
satisfying the same properties as in the proposition, 
except that the inequality in the conclusion becomes
$$C^{-2(r+s)} \modulus{u}_{\mgt} \leq \modulus{u}_\mg \leq C^{2(r+s)} \modulus{u}_\mgt.$$
\end{remark}

Using this operator $\B$, we express the induced volume measures
with respect to each other. 

\begin{proposition}
\label{Prop:MeasRep}
The measure
$d\mu_\mg(x) = \sqrt{\det \B(x)}\ d\mu_\mgt(x)$
for $x$-a.e. and
$C^{-\frac{n}{2}} \mu_\mg \leq \mu_\mgt \leq C^{\frac{n}{2}} \mu_\mg.$
\end{proposition}
\begin{proof}
Let $(U,\psi)$ be a chart near $x$. Then,
for $y \in \sR$,  
$$\pullb{\psi^{-1}}d\mu_\mg(y) = \sqrt{\det\mg(\psi(y))}\ d\Leb(y),$$
and letting $G$ denote the matrix of $\mgt$ in these coordinates
and $\tilde{G}$ the coordinates of $\mg$,
we have that $\tilde{G} = \B G$. 
Thus, $\det\mgt = \det \B\ \det\mg$. Therefore, 
it follows that 
$$\pullb{\psi^{-1}}d\mu_\mg(y) 
	= \sqrt{\det\B(\psi(y))}\ \sqrt{\det\mgt(\psi(y))}\ d\Leb(y)
	= \sqrt{\det\B(\psi(y))} \pullb{\psi^{-1}}d\mu_\mgt(y).$$

For the estimate, at a regular point $x \in \sR$, we
can apply Lemma \ref{Lem:Det} to conclude that 
$\C^{-n} \leq \det \B(x) \leq C^n$.
\end{proof}

With the aid of these two observations, we can now 
demonstrate how the $\Lp{p}$ spaces and Sobolev
spaces of two uniformly close rough metrics compare.
For the purposes of readability, 
we write $\uptheta(x) = \sqrt{\det \B(x)}$
from here on. 

\begin{proposition}
\label{Prop:RoughP}
Let $\mg$ and $\mgt$ be two $C$-close rough metrics.
Then,
\begin{enumerate}[(i)]
\item whenever $p \in [1, \infty)$, 
	$\Lp{p}(\Tensors[r,s]\cM,\mg) = \Lp{p}(\Tensors[r,s]\cM, \mgt)$ with
	$$C^{-\cbrac{r+s + \frac{n}{2p}}} \norm{u}_{p,\mgt}   \leq \norm{u}_{p,\mg} \leq C^{r+s + \frac{n}{2p}} \norm{u}_{p, \mgt},$$
\item for $p = \infty$,
	$\Lp{\infty}(\Tensors[r,s]\cM,\mg) = \Lp{\infty}(\Tensors[r,s]\cM, \mgt)$
	with
	$$C^{-(r+s)} \norm{u}_{\infty, \mgt} \leq \norm{u}_{\infty,\mg} \leq C^{r+s} \norm{u}_{\infty,\mgt},$$
\item the Sobolev spaces 
	$\Sob{1,p}(\cM,\mg) = \Sob{1,p}(\cM, \mgt)$
	and $\Sob[0]{1,p}(\cM, \mg) = \Sob[0]{1,p}(\cM, \mgt)$
	with 
	$$C^{-\cbrac{1 + \frac{n}{2p}}} \norm{u}_{\Sob{1,p},\mgt} 
		\leq \norm{u}_{\Sob{1,p},\mg} \leq C^{{1 + \frac{n}{2p}}} \norm{u}_{\Sob{1,p},\mgt},$$
\item the Sobolev spaces
	$\Sob{\extd,p}(\cM,\mg) = \Sob{\extd,p}(\cM,\mgt)$
	and $\Sob[0]{\extd,p}(\cM, \mg) = \Sob[0]{\extd,p}(\cM, \mgt)$
	with 
	$$C^{-\cbrac{n + \frac{n}{2p}}} \norm{u}_{\Sob{\extd, p}, \mgt}
		\leq \norm{u}_{\Sob{\extd,p},\mg}
		\leq C^{{n + \frac{n}{2p}}} \norm{u}_{\Sob{\extd, p}, \mgt},$$
\item the divergence operators satisfy 
	$\divv_{\mg} = \uptheta^{-1}\divv_{\mgt}\uptheta \B$
	and $\divv_{0,\mg} = \uptheta^{-1}\divv_{0,\mgt}\uptheta \B$.
\end{enumerate} 
\end{proposition}
\begin{proof}
To prove (i), by the density of 
$\Ck[c]{\infty}(\Tensors[r,s]\cM)$
in $\Lp{p}(\Tensors[r,s]\cM,\mg)$ and $\Lp{p}(\Tensors[r,s]\cM,\mgt)$, 
it suffices to prove the inequality for
$u \in \Ck[c]{\infty}(\Tensors[r,s]\cM)$.
The $C$-close condition 
implies that $$ C^{-p(r+s)} \modulus{u}_{\mgt}^p \leq \modulus{u}^p_\mg \leq C^{p(r+s)} \modulus{u}_\mgt^p,$$
and Proposition \ref{Prop:MeasRep} gives
$$C^{-\frac{n}{2}}\int_{\cM} \modulus{u}^p_{\mg}\ d\mu_\mg \leq \int_{\cM} \modulus{u}^p_{\mg}\ d\mu_\mgt 
	\leq C^{\frac{n}{2}}\int_{\cM} \modulus{u}^p_{\mg}\ d\mu_\mg.$$
Combining the two proves (i).

To prove (ii), note that
if $C' > 0$ such that $\modulus{u(x)}_{\mg(x)} < C'$ $x$-a.e., 
then $\modulus{u(x)}_{\mgt(x)} \leq C' C^{r+s}$ $x$-a.e.
Therefore, $\norm{u}_{\infty,\mg} \leq C^{r+s} \norm{u}_{\infty,\mgt}$.
It is easy to see that 
the same is true with $\mg$ and $\mgt$ interchanged.

Fix $u \in \Ck{\infty} \intersect \Lp{p}(\cM)$ 
such that $\conn u \in \Ck{\infty}\intersect \Lp{p}(\cotanb\cM)$
(we omit the $\mg$ and $\mgt$ dependence of $\Lp{p}$ since by (i), the $\Lp{p}$ spaces are the same under both $\mg$ and $\mgt$).
Then,
$$\norm{u}_{\Sob{1,p},\mg} 
	= \norm{u}_{p,\mg} + \norm{\modulus{u}}_{p,\mg}
	\leq C^{\frac{n}{2p}} \norm{u}_{p,\mgt} + C^{\cbrac{1 + \frac{n}{2p}}} \norm{\conn u}_{p,\mgt}
	\leq C^{\cbrac{1 + \frac{n}{2p}}} \norm{u}_{\Sob{1,p},\mgt}.$$ 
Similarly, we can interchange $\mg$ and $\mgt$ to obtain 
the lower inequality which proves (iii).

The claim in (iv) is proven similarly, on noting that
$\Forms(\cM) \subset \oplus_{s=0}^n \Tensors[s,0](\cM)$
and thus the largest constant  that
can appear is $C^n$.

To prove the last claim,
let $u \in \dom(\close{\conn_2}) = \Sob{1,2}(\cM)$
and $v \in \dom(\divv_\mg)$. By construction of
$\divv_\mg$, we have that
$\inprod{\close{\conn_2}u,v}_\mg = \inprod{u,-\divv_\mg v}_\mg$.
By what we have already established, we
have that
$\inprod{\close{\conn_2}u,v}_\mg
	= \inprod{\uptheta \B \close{\conn_2}u,v}_\mgt 
	= \inprod{\close{\conn_2}u, \uptheta\B v}_\mgt$
and
$\inprod{u,\divv_\mg v}_\mg = \inprod{u, \uptheta \divv_\mg v}_\mgt$.
Combining these two calculations,
we obtain that 
$\inprod{\close{\conn_2}u, \uptheta \B v}_\mgt =
	\inprod{u, -\uptheta \divv_\mg v}_\mgt$, and therefore,
$\uptheta \B v \in \dom(\divv_\mgt)$, 
$\dom(\divv_\mg) \subset \dom(\divv_\mgt \uptheta \B)$
and 
$\uptheta^{-1} \divv_\mgt \uptheta \B v = \divv_\mg v$.
For the reverse inclusion, let $v \in \dom(\divv_\mgt\uptheta \B)$
and then, 
$\inprod{u, -\uptheta^{-1} \divv_\mgt \uptheta \B v}_\mg 
	= \inprod{u, -\divv_\mgt \uptheta \B v}_\mgt
	= \inprod{\close{\conn_2}u, \theta \B v}_\mgt
	= \inprod{\close{\conn_2}u, v}_\mg.$
Hence $v \in \dom(\divv_\mg)$
with $\divv_\mg v = \uptheta^{-1} \divv_\mgt \uptheta \B v$.

By replacing $\conn_2$ by $\conn_0$ and hence $\divv_{\mg}$
and by $\divv_{0,\mg}$ and $\divv_{\mgt}$ by $\divv_{0,\mgt}$
proves $\divv_{0,\mg} = \uptheta^{-1} \divv_{0,\mgt}\uptheta\B$.
\end{proof}

\begin{remark}
\begin{enumerate}[(i)]
\item Note that in (iii) and (iv), we only consider
Sobolev spaces over functions or
the exterior derivative. 
The exterior derivative depends
only on the topology of $\cM$ 
and it is independent of the metric.
An attempt to prove similar
results for tensors is a futile
effort (at least using these methods)
for the simple fact that 
the Levi-Cevita connection depends on the
metric. 
In fact, we do not even know what we mean 
by a Levi-Cevita connection
if the metric is of regularity less than $\Ck{1}$.
 
\item Defining $\divv_\mg$ abstractly 
as the negative of the adjoint of $\conn_2$
prevents us from knowing whether  $\divv_\mg$ is 
even a differential operator. 
Indeed,
for a smooth metric $\mh$ and a corresponding
compatible connection $\conn^{\mh}$,
we have  
$\inprod{ \conn^{\mh} u, v} = \inprod{u, -\tr\conn^{\mh}v}$
whenever $u \in \Ck[c]{\infty}(\cM)$
and $v \in \Ck[c]{\infty}(\cotanb\cM)$, 
so that $\divv_\mh \supset -\tr \conn^\mh$.
It is a lack of such a compatibility formula
for rough metrics which prevents 
us from knowing the differential 
properties of $\divv_\mg$. However, 
when $\mgt$ is a smooth metric,
(v) illustrates that
$\divv_\mg$ is indeed the differential operator
$\divv_\mgt$ but with measurable coefficients.
\end{enumerate}
\end{remark}

\subsection{There are smooth geometries near continuous ones}

In this final section, 
we consider the special case of continuous metrics
as rough metrics.
For such geometries, we establish that there
are always $C$-close smooth geometries
for any choice of $C > 1$.
As a consequence, 
 we show that 
the operators $\extd_2$ (and hence, $\extd_0$)
are closable operators
for every continuous geometry, 
 a result we promised we 
would prove 
with the aid of better tools 
in \S\ref{Sect:Prelim:Mfld}.
 
\begin{lemma}
\label{Lem:CtsCoords}
Let $\mg$ be a continuous metric. Then, for each 
$x \in \cM$, there exists an $r_x > 0$ 
such that $B(x,r_x)$ is a smooth coordinate
system and at $x$, the metric in the coordinate
directions $\set{\partial_i}$ satisfy
$\mg_{ij}(x) = \mg_x (\partial_i, \partial_j) = \delta_{ij}$.
\end{lemma}
\begin{proof}
Fix $x \in \cM$ and let $\set{e_j}$ be an orthonormal 
basis with respect to $\mg(x)$ for $\tanb_x \cM$.
This exists via a Gram-Schmidt process since
$\mg$ is non-degenerate.

Next, we note that the manifold $\cM$ admits a connection
$\conn$ because it admits
a smooth metric $\mgt$ and we can
take $\conn$ to be the Levi-Cevita connection 
of this metric. We consider
the exponential map, 
$\exp_x: \tanb_x\cM \to \cM$, with respect to $\conn$.

Now, given the basis $\set{e_j}$ at $x$, we are
able to obtain a neighbourhood $U \subset \tanb_x \cM$
for which $\exp_x:U \to \cM$ is a smooth injection
since it is the exponential map of the smooth metric $\mgt$.
Thus, by identifying $\tanb_x \cM$ with $\R^n$,
we obtain a coordinate system $U'$ near $x$
such that $\partial_j(x) = e_j$. Now we pick
$r_x$ such that $B(x,r_x) \subset U'$.
\end{proof}

Using this Lemma, we can prove the following 
result that asserts that there are smooth 
geometries arbitrarily close to continuous ones.

\begin{proposition}
\label{Prop:ContSmoothMet}
Let $\mg$ be a continuous metric. 
Given any $C > 1$, there exists 
a smooth metric $\mgt$ which is $C$-close
to $\mg$.
\end{proposition}
\begin{proof}
Fix $x \in \cM$. As a consequence
of Lemma \ref{Lem:CtsCoords}, there exists
an $r_x$ and a coordinate system 
$(B(x,r_x),\psi_x)$ such that 
at $x$, $\modulus{u}_{\mg} = \modulus{u}_{\pullb{\psi}_x\delta}$
for all $u \in \tanb_x\cM$.

Let $v \in \Ck{\infty}(\tanb B(x,r_x'))$, we can find
$r_x' \leq r_x$ such that 
$$\frac{1}{C} \modulus{v(y)}_{\pullb{\psi}_x\delta}
	\leq \modulus{v(y)}_{\mg(y)} 
	\leq C \modulus{v(y)}_{\pullb{\psi}_x\delta}$$
for all $y \in B(x, r_x')$.
Since we only have $n$ independent directions,
we can obtain $r_x'$ independent of
$v$.

Now, let $B_i$ be a countable
subcover of $\set{B(x,r_x')}$, 
and let $\set{\phi_i}$ be a smooth
partition of unity subordinate to
$\set{B_i}$. Then, for $x \in \cM$
define 
$$\mgt(x) = \sum_{i=1}^\infty \phi_i(x) (\pullb{\psi}_i\delta)(x).$$
Indeed, each of our coordinate systems $(B_i,\psi_i)$ 
are smooth, and hence, $\mgt$ is smooth. 
By construction, we have that
$\mgt$ is $C$-close to $\mg$.
\end{proof}

As promised, we illustrate the following
immediate consequence.

\begin{corollary}
Let $\mg$ be a continuous metric. 
Then the operators $\extd_p:\Ck{\infty}\intersect \Lp{p}(\Forms(\cM))
\to \Ck{\infty}\intersect \Lp{p}(\Forms(\cM))$ and
$\extd_0:\Ck[c]{\infty}(\Forms(\cM)) \to \Ck[c]{\infty}(\Forms(\cM))$
are closable, densely-defined operators. 
Moreover, $\Sob{\extd,p}(\cM) = \dom(\close{\extd_p})$
and $\Sob[0]{\extd,p}(\cM) = \dom(\close{\extd_0})$.
\end{corollary}
\begin{proof}
By the previous proposition, there exists
a smooth metric $\mgt$ that is $2$-close everywhere 
to $\mg$.
The conclusion then follows
from Proposition \ref{Prop:RoughP}.
\end{proof}

For the $p = 2$ case, 
this asserts that
the adjoints $\intd_p = \adj{\extd_p}$
and $\intd_0 = \adj{\extd_0}$ 
both exist as densely-defined, 
closed operators. Furthermore, we 
immediately obtain the following as
a corollary.

\begin{corollary}
\label{Cor:CtsComp}
Let $\mg$ be a continuous, complete metric. 
Then $\close{\extd}_p = \close{\extd}_0$ and 
$\intd_0 = \intd_p$.
\end{corollary}
\begin{proof}
Again, we obtain a smooth $2$-close everywhere
metric $\mgt$, which is guaranteed to 
be complete since $\mg$ is complete.
That the conclusion holds for a smooth
metric is well known fact in the folklore.
\end{proof}
\section{Reducing low regularity problems to smooth ones}
\label{Sect:Red}

The goal of this section is to 
demonstrate the reduction of low regularity
Kato square root
problems to smooth ones via quadratic estimates. 
We first establish
a general framework at the level of Hilbert spaces. 
We then apply this technology, along with the results we 
have previously obtained, to illustrate how to pass
quadratic estimates for the Kato square root problem
on functions
between two uniformly close metrics, 
as well as  a similar problem for inhomogeneous Hodge-Dirac
operators as considered by the author
in \S6 of \cite{BThesis}.

\subsection{Reduction at the level of the AKM framework}

We begin by describing a general framework
that encapsulates the reduction we consider later.  The 
primary feature of our viewpoint, borrowed
from \cite{AKMc}, is that one
part of the operator remains fixed while
the other part changes under a change of inner
product.

Recall from \S\ref{Sect:AKM}
that $\Gamma: \dom(\Gamma) \subset \Hil \to \Hil$ is
a closed, densely-defined and nilpotent operator.
Let $\inprod{\mdot,\mdot}_1$ and $\inprod{\mdot,\mdot}_2$
be two inner products on $\Hil$ and suppose
there exists a bounded, self-adjoint, invertible
operator $\Phi \in \bddlf(\Hil)$
such that $\inprod{u,v}_1 = \inprod{\Phi u,v}_2$ 
for all $u, v\in \Hil$.

Let $\adj{\Gamma}_1$ denote the adjoint of 
$\Gamma$ with respect to $\inprod{\mdot,\mdot}_1$
and similarly $\adj{\Gamma}_2$ denote the
adjoint of $\Gamma$ with respect to $\inprod{\mdot,\mdot}_2$.
We first obtain the following transformation rule
whose proof is similar to the proof of case (v) of Proposition 
\ref{Prop:RoughP}.

\begin{proposition}
\label{Prop:GammaChange}
The operator $\adj{\Gamma}_1 = \Phi^{-1}\adj{\Gamma}_2 \Phi.$
\end{proposition}

Let $B_i \in \bddlf(\Hil)$ satisfying (H2) and (H3)
of the AKM framework from \S\ref{Sect:AKM}
with respect to $\inprod{\mdot,\mdot}_1$.
By the previous proposition, we write
$$\Pi_{B,1} = \Gamma + B_1 \adj{\Gamma}_1 B_2 
	= \Gamma + B_1 \Phi^{-1} \adj{\Gamma}_2 \Phi B_2.$$
Thus, we define 
$\Pi_{\tilde{B},2} = \Gamma + \tilde{B}_1 \adj{\Gamma}_2 \tilde{B}_2$
where $\tilde{B_1} = B_1 \Phi^{-1}$ and $\tilde{B}_2 \Phi$
so that $\Pi_{B,1} = \Pi_{\tilde{B},2}$.
We prove later (under a very mild additional assumption)
that $\Pi_{B,1}$ satisfies quadratic estimates
if and only if $\Pi_{\tilde{B},2}$ satisfies
quadratic estimates. 

Recall that in 
the AKM framework, the $B_1$ satisfies an 
accretivity assumption with respect to $\ran(\adj{\Gamma})$. 
In our situation, we need to tweak this assumption
to reflect the fact that we have
two adjoint operators arising from the
two different inner products.
We present this modification as (H2') below.

\begin{enumerate}[(H1)']
\item[(H2')] Suppose there exists $\adj{\sR} \subset \Hil$
	such that $\Phi \adj{\sR} = \adj{\sR}$,
	$\ran(\adj{\Gamma}_1) \union \ran(\adj{\Gamma}_2) \subset \adj{\sR}$,
	and
	$$\re\inprod{B_1u, u}_1 \geq \kappa_1 \norm{u}^2_1,
	\ \text{and}\ 
	\re\inprod{B_2v,v}_1 \geq \kappa_2 \norm{v}^2_1,$$ 
	for all $u \in \adj{\sR}$ and $v \in \ran(\Gamma)$.
\end{enumerate}

We remark that in practise, the space
$\adj{\sR}$ arises naturally. 

It is unreasonable to expect that 
changing from the operator $\Pi_{B,1}$ to 
$\Pi_{\tilde{B},2}$ can be done for free. 
The following proposition
calculates the cost we must pay in accretivity.

\begin{proposition}[Cost in accretivity for change of operator]
\label{Prop:AccCost}
Suppose that
\begin{align*}
&\norm{\Phi^{-\frac{1}{2}}u}^2_2 = \inprod{\Phi^{-1} u, u}_2 \geq \eta_1\norm{u}_2^2\\
&\norm{\Phi^{\frac{1}{2}}v}^2_2 = \inprod{\Phi v, v}_2 \geq \eta_2 \norm{v}_2^2  
\end{align*}
for $u \in \adj{\sR}$ and $v \in \ran(\Gamma)$. Then, assuming
that (H1), (H2') and (H3) are satisfied for $\Gamma$ and 
$B_i$ in $\inprod{\mdot,\mdot}_1$ with constants $\kappa_i$, 
then the same assumptions are satisfied for $\tilde{B}_i$ in 
$\inprod{\mdot,\mdot}_2$.
The constants in (H2') are then $\kappa_1\eta_1$ and $\kappa_2\eta_2$.
\end{proposition}
\begin{proof}
It is an easy observation that (H1) is satisfied in $\inprod{\mdot,\mdot}_2$. 

Let us first consider (H3).
We note that
$\tilde{B}_1 \tilde{B}_2 = B_1 \Phi^{-1} \Phi B_2 = B_1B_2$.
Hence, it is trivial that 
$\tilde{B}_1 \tilde{B}_2 \ran(\Gamma) = B_1 B_2 \ran(\Gamma) \subset \nul(\Gamma)$.
Next, note that $\tilde{B}_2 \tilde{B}_1  = \Phi B_2 B_1 \Phi^{-1}$.
Also, as a consequence of Proposition \ref{Prop:GammaChange},
we have that $\adj{\Gamma}_2 = \Phi \adj{\Gamma}_1 \Phi^{-1}$.
Thus, 
$$\tilde{B}_2 \tilde{B}_1 \adj{\Gamma}_2 
	= \Phi B_2 B_1 \Phi^{-1} \Phi \adj{\Gamma}_1 \Phi
	= \Phi B_2 B_1 \adj{\Gamma}_1 \Phi
	= 0.$$
Thus, $\tilde{B}_2 \tilde{B}_1 \ran(\adj{\Gamma}_2) \subset \nul(\adj{\Gamma}_2)$.

Now, we show that (H2') is satisfied for $\Gamma$ and $\tilde{B}_i$
in $\inprod{\mdot,\mdot}_2$. 
Let us fix $u \in \adj{\sR}$ and note that
$\inprod{\tilde{B}_1 u, u}_2
	= \inprod{B_1 \Phi^{-1}u, u}_2.$
Let $u' = \Phi^{-1}u$ and $u' \in \adj{\sR}$ by assumption.
Then, 
$$\inprod{\tilde{B}_1u,u}_2 
	= \inprod{B_1u', \Phi u'}_2
	= \inprod{B_1u', u'}_1
	\geq \kappa_1\norm{u'}_1^2 
	= \geq \kappa_1 \norm{\Phi^{-1}u}_1^2.$$
But we have by assumption on $\inprod{\mdot,\mdot}_1$
and $\inprod{\mdot,\mdot}_2$ that
$\norm{u}^2_1 = \inprod{u,u}_1 = \inprod{\Phi u, u}_2 = \inprod{\Phi^{\frac{1}{2}}u}^2.$
Thus,
$\norm{\Phi^{-1}u}^2_1 = \norm{\Phi^{-\frac{1}{2}}u}_2 \geq \eta_1 \norm{u}_2$.
This proves that 
$\inprod{\tilde{B}_1u,u}_2 \geq \kappa_1 \eta_1 \norm{u}_2^2.$

Next, let $v \in \ran(\Gamma)$. Then, 
$$\inprod{\tilde{B}_2v,v}_2
	= \inprod{\Phi B_2v,v}_2
	= \inprod{ B_2 v,v}_1
	\geq \kappa_2 \norm{v}_1^2 
	= \kappa_2 \norm{\Phi^{\frac{1}{2}}v}_2^2
	\geq \kappa_2 \eta_2 \norm{v}_2^2,$$
which finishes the proof.
\end{proof}

The main tool that we shall require in later
sections is the following.

\begin{proposition}
\label{Prop:Reduction}
The quadratic estimate
$$ \int_0^\infty \norm{t \Pi_{B,1}(1 + t^2 \Pi_{B,1}^2)^{-1}u}^2_1\ 
	\frac{dt}{t} \simeq \norm{u}^2_1$$
is satisfied  for all $u \in \close{\ran(\Pi_{B,1})}$
if and only if 
$$\int_0^\infty \norm{t \Pi_{\tilde{B},2}(1 + t^2 \Pi_{\tilde{B},2})^{-1}v}^2_2\ 
	\frac{dt}{t} \simeq \norm{v}^2_2$$ 
is satisfied for all $v \in \close{\ran(\Pi_{B,2})}$.
\end{proposition}
\begin{proof}
It suffices to note that $\ran(\Pi_{\tilde{B},2}) = \ran(\Pi_{B,1})$
and that $\norm{\mdot}_1 \simeq \norm{\mdot}_2$.
\end{proof}

\subsection{The Kato square root problem for functions}
\label{Sect:RedFn}

With the aid of this framework, we first 
consider the Kato square root problem for functions.
Let $\mg$ and $\mgt$ be two $C$-close metrics, 
and suppose that $\mgt$ is at least continuous and complete.
Let $\Hil = \Lp{2}(\cM) \oplus \Lp{2}(\cM) \oplus \Lp{2}(\cotanb\cM)$
and let $\inprod{\mdot,\mdot}_1 = \inprod{\mdot,\mdot}_\mg$
and $\inprod{\mdot,\mdot}_2 = \inprod{\mdot,\mdot}_\mgt$.
Define $\Phi: \Hil \to \Hil$ by 
 $\Phi(u,v,w) = (\uptheta u, \uptheta v, \uptheta \B w)$. It is
easy to see that $\Phi \in \bddlf(\Hil)$, 
symmetric, positive, invertible and that 
$\inprod{u,v}_\mg = \inprod{\Phi u, v}_{\mgt}$.

By the assumption of continuity and completeness on $\mgt$, 
we conclude from Proposition \ref{Prop:OpSob} 
that $\close{\conn[0]} = \close{\conn[2]}$
and $\divv_{0,\mgt} = \divv_{\mgt}$.
Write $S = (\iden, \close{\conn}_2)$
and fix $a \in \Lp{\infty}(\cM)$ and 
$A \in \Lp{\infty}(\bddlf(\Lp{2}(\cM) \oplus \Lp{2}(\cotanb\cM)))$
such that the following ellipticity assumption holds: 
there exists $\kappa_1, \kappa_2 > 0$ such that
\begin{align*}
\label{Ass:Ell}
\tag{$\text{E}_\mg$}
\re\inprod{a u, u}_\mg \geq \kappa_1 \norm{u}_\mg^2
\ \text{and}\ \re\inprod{A Sv, Sv}_\mg \geq \kappa_2 \norm{v}^2_{\Sob{1,2},\mg}
\end{align*}
for $u \in \Lp{2}(\cM,\mg)$ and $v \in \Sob{1,2}(\cM,\mg)$.

We recall that the Kato square root problem for functions
is then to determine whether the following holds:
\begin{align*}
\tag{$\text{K}_{S,\mg}$}
\label{Ass:KatoFn}
\dom(\sqrt{a\adj{S}AS}) = \Sob{1,2}(\cM,\mg)
\ \text{with}\ 
\norm{\sqrt{a\adj{S}AS}u}_{\mg} \simeq \norm{u}_{\Sob{1,2},\mg}
\end{align*}
for $u \in \Sob{1,2}(\cM,\mg)$.

As outlined in \S\ref{Sect:AKM}, we let
$$
\Gamma = \begin{pmatrix}0 & 0 \\ S & 0 \end{pmatrix},\ 
B_1 = \begin{pmatrix} a & 0 \\ 0 & 0 \end{pmatrix}, 
\ \text{and}\ 
B_2 = \begin{pmatrix} 0 & 0 \\ 0 & A \end{pmatrix},$$
and $\adj{\sR} = \Lp{2}(\cM) \oplus 0 \oplus 0$. Then, we note that
\eqref{Ass:Ell} is equivalent to 
$$\re\inprod{B_1u, u}_\mg \geq \kappa_1 \norm{u}_\mg^2
\ \text{and}\  \re\inprod{B_2v, v}_\mg \geq \kappa_2 \norm{v}_\mg^2$$
whenever $u \in \adj{\sR}$ and $v \in \ran(\Gamma)$.
It is also easy to see that 
$\Phi \adj{\sR} = \adj{\sR}$, 
$\ran(\adj{\Gamma}_\mg) \union \ran(\adj{\Gamma}_\mgt) \subset \adj{\sR}$
and that 
$$ 
\tilde{B}_1 = \begin{pmatrix} a \uptheta^{-1} & 0  \\ 0 & 0 \end{pmatrix},
\ \text{and}\ 
\tilde{B}_2 = \begin{pmatrix} 0 & 0 \\ 0 & T A \end{pmatrix},$$
where $T:\Lp{2}(\cM) \oplus \Lp{2}(\cotanb\cM) \to \Lp{2}(\cM) \oplus \Lp{2}(\cotanb\cM)$ via $T(u,v) = (\uptheta u, \uptheta \B v)$.

As an immediate consequence to Proposition \ref{Prop:AccCost},
we obtain the following.

\begin{proposition}
\label{Prop:ChangeAcc}
The operators $\tilde{B}_i \in \bddlf(\Hil)$ satisfy
\begin{align*}
&\re\inprod{\tilde{B}_1u, u}_\mgt \geq \frac{\kappa_1}{C^{\frac{n}{2}}}\norm{u}_\mgt^2,\\
&\re\inprod{\tilde{B}_2v, v}_\mgt \geq \frac{\kappa_2}{C^{1 + \frac{n}{2}}} \norm{v}_\mgt^2
\end{align*}
whenever $u \in \adj{\sR}$ and $v \in \ran(\Gamma)$.
\end{proposition}
\begin{proof}
By Proposition \ref{Prop:AccCost}, it suffices
to simply compute lower bounds
for $\norm{\Phi^{-\frac{1}{2}}u }_2$ and $\norm{\Phi^{\frac{1}{2}} v}_2$
for appropriate $u$ and $v$.

First, fix $u \in \adj{\sR}$. Then,
$$\norm{\Phi^{-\frac{1}{2}}u}_2^2 
	= \norm{\theta^{-\frac{1}{2}} u_1}_2^2
	\geq C^{-\frac{n}{2}} \norm{u_1}_2^2 
	= C^{-\frac{n}{2}} \norm{u}_2^2.$$ 
Next, let $v \in 0 \oplus \Lp{2}(\cM) \oplus \Lp{2}(\cotanb\cM) \supset \ran(\Gamma)$.
Then,
$$\norm{\Phi^{\frac{1}{2}} v}_2^2
	= \norm{\uptheta^{\frac{1}{2}} v_2}^2_2 + \norm{\uptheta\B v_3}_2^2
	\geq C^{-\frac{n}{2}} \norm{v_2}^2_2 + C^{-\cbrac{1 + \frac{n}{2}}} \norm{v_3}_2^2
	\geq C^{-\cbrac{1 + \frac{n}{2}}} \norm{v}_2^2$$
which finishes the proof.
\end{proof}

Combining these results, we obtain the following main theorem of
this section as a consequence of Proposition \ref{Prop:Reduction}.

\begin{theorem}
Let $\mg$ be a rough metric and $\mgt$ continuous, complete
and suppose that they are uniformly close. Further 
suppose that
$$ \int_0^\infty \norm{t \Pi_{\tilde{B},\mgt}(1 + t^2 \Pi_{\tilde{B},\mgt})^{-1}u}_\mgt\ 
	\frac{dt}{t} \simeq \norm{u}_\mgt^2$$
for all $u \in \close{\ran(\Pi_{\tilde{B},\mgt})}$.
Then, 
$$ \int_0^\infty \norm{t \Pi_{{B},\mg}(1 + t^2 \Pi_{{B},\mg})^{-1}u}_\mg\ 
	\frac{dt}{t} \simeq \norm{u}_\mg^2$$
for all $u \in \close{\ran(\Pi_{B,\mg})}$.
\end{theorem}

By combining this with Theorem 1 of McIntosh 
and the author in \cite{B3}, we obtain the following
important corollary. 
\begin{corollary}
\label{Cor:KatoFn}
Let $\mgt$ be a smooth, complete metric and suppose that
there exists $\kappa> 0$ and $\eta > 0$ such that
$\inj(\cM,\mgt) \geq \kappa$ and $\Ric(\mgt) \leq \eta$.
Then, for any rough metric $\mg$ that is 
uniformly close, quadratic estimates are satisfied
for $\Pi_{B,\mg}$. 
In particular \eqref{Ass:KatoFn} holds
under the assumption \eqref{Ass:Ell}.
\end{corollary}

\subsection{The Kato square root problem for differential forms}
\label{Sect:RedD}

In his thesis \cite{BThesis}, the author considers
versions of the Kato square root problem for perturbations
of inhomogeneous Hodge-Dirac operators under
a natural and mild curvature 
assumption on the bundle of forms.
 
Let $\mg$ be a rough metric and 
$A \in \Lp{\infty}(\bddlf(\Lp{2}(\Forms(\cM)) \oplus \Lp{2}(\Forms(\cM))))$.
We assume that $A$ satisfies the following ellipticity condition
with respect to $\mg$: there exists $\kappa_2 > 0$ such that 
\begin{equation*}
\label{Ass:EllDirac}
\tag{$\text{E}_{\Dir,\mg}$}
\re\inprod{Au, u}_\mg \geq \kappa_2 \norm{u}_\mg^2 
\end{equation*}
for every $u \in \Lp{2}(\Forms(\cM)))$.
Indeed, we immediately obtain that 
there exists $\kappa_1 > 0$ such that 
$\re{\inprod{A^{-1}u, u}}_\mg \geq \kappa_1 \norm{u}_\mg^2$.

Let $\Dir_{A,\mg} = \extd + A^{-1}\intd_\mg A$.
Given some $0 \neq \upbeta \in \C$, 
the Kato square root problem for forms 
as outlined in \S6.4 in \cite{BThesis}
is then to determine that 
\begin{align*}
\tag{$\text{K}_{\Dir, \mg}$}
\label{Ass:KatoDir}
&\dom(\sqrt{\Dir_{A,\mg}^2 + \modulus{\upbeta}^2}) = \dom(\Dir_{A,\mg}), \\
&\norm{\sqrt{\Dir_{A,\mg}^2 + \modulus{\upbeta}^2}u} \simeq 
	\norm{\Dir_{A,\mg} u} + \norm{u} 
	\simeq \norm{\extd u} + \norm{\intd_\mg A u} + \norm{u}
\end{align*}
for $u \in \dom(\Dir_{A,\mg})$.
 
Now, let $\mgt$ be
at least continuous and complete, and assume that 
it is $C$-close to $\mg$.
We denote the induced canonical 
metrics on $\Forms(\cM)$ by the same symbols. 
As we have noted previously, 
for almost-every $x \in \cM$ and every $u \in \Forms_x(\cM)$, we
have the inequality
$$C^{-n} \modulus{u}_{\mgt(x)} \leq \modulus{u}_{\mg(x)} \leq C^n \modulus{u}_{\mgt(x)}.$$
An argument along the lines of 
the proof of
Proposition \ref{Prop:OpExist} guarantees an operator
$\E: \Sect(\adj{\Forms}(\cM) \tensor \Forms(\cM))$ such that
$\mg_x(u,v) = \mgt_x (\E(x)u,v)$ 
satisfying the inequality
$$\C^{-2n} \modulus{u}_{\mg(x)} \leq \modulus{\E(x)u}_{\mg(x)} 
	\leq C^{-2n} \modulus{u}_{\mg(x)}$$ 
for almost-every $x \in \cM$.
 
Let $\Hil = \Lp{2}(\Forms(\cM)) \oplus \Lp{2}(\Forms(\cM))$
and note that $\mg$ and $\mgt$ induces
$\inprod{\mdot,\mdot}_\mg$ and $\inprod{\mdot,\mdot}_\mgt$
respectively. On setting 
$\Phi(w,z) = (\uptheta \E w, \uptheta \E z)$,
we can see that
$\inprod{u,v}_\mg = \inprod{\Phi u, v}_\mgt$. 

To encode the problem into a Dirac-type operator,
fix $\upbeta \in \C$ with $\upbeta \neq 0$ and
let
$$ 
\extd_\upbeta = \begin{pmatrix} \extd & 0 \\ \upbeta & -\extd \end{pmatrix}.$$
The operator $\extd$ here is the operator $\close{\extd_0}$ or $\close{\extd_2}$, 
which are equal in both metrics $\mg$ and $\mgt$ as
a consequence of Corollary \ref{Cor:CtsComp}, 
the continuity and completeness of $\mgt$, 
and the $C$-closeness
of the two metrics.
The adjoint of $\extd_\upbeta$ with respect to $\mg$ and
$\mgt$ are denoted by $\intd_{\upbeta,\mg}$ and $\intd_{\upbeta,\mgt}$
respectively. It is easy to see that these are given by the operator matrices
$$
\intd_{\upbeta,\mg} = \begin{pmatrix} \intd_{\mg} & \conj{\beta} \\ 0 & \intd_{\mg} \end{pmatrix}
\ \text{and}\ 
\intd_{\upbeta,\mgt} = \begin{pmatrix} \intd_{\mgt} & \conj{\beta} \\ 0 & \intd_{\mgt} \end{pmatrix}.$$

By repeating the argument proving (v) of Proposition \ref{Prop:RoughP}, 
we obtain that $\intd_{\mg} = (\E\theta)^{-1} \intd_{\mgt} (\E\theta)$
and also that $\intd_{\beta,\mg} = \Phi^{-1} \intd_{\beta,\mgt} \Phi$.

Next, define $B_1, B_2 \in \bddlf(\Hil)$ by
$$B_1 = \begin{pmatrix} A^{-1} & 0 \\ 0 & A^{-1} \end{pmatrix}
\ \text{and}\ B_2 =  \begin{pmatrix} A & 0 \\ 0 & A \end{pmatrix}.$$
On setting $\adj{\sR} = \Hil$, by the ellipticity assumption on 
$A$, we obtain that
\begin{align*}
\re\inprod{B_1u,u}_\mg \geq \kappa_1 \norm{u}_\mg^2
\ \text{and}\ \re\inprod{B_2u,u}_\mg \geq \kappa_2 \norm{u}_\mg^2
\end{align*}
for all $u \in \Hil$.
Recall the operator $\Pi_{B,\mg} = \Gamma + B_1 \adj{\Gamma}_\mg B_2$
from the AKM framework and note that
$$\Pi_{B,\mg}^2 = \begin{pmatrix} 
	\Dir_{A,\mg}^2 + \modulus{\upbeta} & 0 \\
	0 & \Dir_{A,\mg}^2 + \modulus{\upbeta}
	\end{pmatrix}.$$
It is for the operator $\Pi_{B,\mg}$ for which 
we consider quadratic estimates to ultimately
yield a solution to \eqref{Ass:KatoDir}.
		
As in \S\ref{Sect:RedFn}, we reduce the non-smooth problem
to a smooth one. It is easy to see that
$\Pi_{B,\mg} = \Pi_{\tilde{B},\mgt}$
where $\tilde{B}_1 = B_1\Phi^{-1}$
and $\tilde{B}_2 = \Phi B_2$.
By applying a similar argument
as in the proof of Proposition \ref{Prop:ChangeAcc},
we obtain the following change of accretivity
in moving from $\Pi_{B,g}$ to $\Pi_{\tilde{B},\mgt}$.
\begin{proposition}
The operators $\tilde{B}_1$ and $\tilde{B}_2$ 
satisfy 
$$\re\inprod{\tilde{B}_i u, u}_\mgt \geq \frac{\kappa_i}{C^{\frac{3n}{2}}} \norm{u}_\mgt$$
for $u \in \Lp{2}(\Forms(\cM))$ and $i = 1,2$.
\end{proposition}
\begin{proof}
As a consequence of Proposition \ref{Prop:AccCost}, it suffices
to compute a lower bound for $\norm{\Phi^{\frac{1}{2}} u}_\mgt$
for $u \in \Lp{2}(\Forms(\cM))$.
But it is easy to observe that
$\norm{\Phi^{\frac{1}{2}} u}_\mgt \geq  {C^{n\cbrac{1 + \frac{1}{2}}}} \norm{u}_\mgt$. 
\end{proof} 

Recall that for a smooth, complete metric $\mgt$,
the curvature endomorphism
$\Rend: \Forms_x(\cM) \to \Forms_x(\cM)$ 
is given by
$$ \Rend \omega = -\Rm_{ijkl}(x)\  \theta^i \wedge (\theta^j \cut (\theta^l \wedge (\theta^k \cut \omega))),$$
where $\set{\theta^i}$ are an orthonormal frame at $x$ and 
$\omega \in \Forms_x(\cM)$.
In Theorem 6.4.3 of \cite{BThesis}, the author
shows that the quadratic estimates for 
$\Pi_{\tilde{B}, \mgt}$ are satisfied under appropriate
bounds on the geometry of $\mgt$ and on $\Rend$.
Coupling this result with Proposition \ref{Prop:Reduction},
we have the following main theorem of this section. 

\begin{theorem}
\label{Thm:KatoD}
Let $\mg$ be a rough metric uniformly close to $\mgt$, a smooth, complete metric,
and suppose that:
\begin{enumerate}[(i)]
\item there exists $\kappa > 0$ such that $\inj(\cM,\mgt) \geq \kappa$, 
\item there exists $\eta > 0$ such that $\modulus{\Ric(\mgt)} \leq  \eta$, and
\item there exists $\zeta \in \R$ such that 
$ \mgt(\Rend\omega,\omega) \geq \zeta \norm{\omega}^2_\mgt.$
\end{enumerate}
Then, whenever $A \in \Lp{\infty}(\bddlf(\Lp{2}(\Forms\cM)))$ satisfies
\eqref{Ass:EllDirac}, we obtain \eqref{Ass:KatoDir}. 
\end{theorem} 

\subsection{Applications to compact manifolds}

In this short section, we present the following theorem which is a culmination of results we
have obtained so far. It demonstrates that the 
aforementioned Kato square root
problems can always be solved on compact 
manifolds for \emph{every} rough metric.

\begin{theorem}
\label{Thm:CpctContMet}
Let $\cM$ be a smooth, compact manifold
and $\mg$ a rough metric. Whenever $A$ satisfies 
\eqref{Ass:Ell} then \eqref{Ass:KatoFn} holds.
Similarly, if $A$ satisfies
\eqref{Ass:EllDirac} then \eqref{Ass:KatoDir} holds.
\end{theorem}
\begin{proof}
By the compactness of $\cM$, we have
a finite number of charts $(U_i, \psi_i)$
covering $\cM$ satisfying the local comparability
condition with constants $C_i$. On letting
$\phi_i$ be a smooth partition of unity subordinate
to $\set{U_i}$, write 
$$\mgt(x) = \sum_{i} \phi_i \pullb{\psi}_i\delta(x).$$
It is easy to see that $\mgt$ is a smooth metric.
On setting $C = \max_i\set{C_i}$, we obtain
that $\mg$ and $\mgt$ are $C$-close. 

Since $\modulus{\Ric_\mgt}_\mgt: \cM \to \R$
is smooth and in particular continuous, 
by the compactness of $\cM$, there is an $\eta > 0$ and
$\zeta \in \R$
such that $\modulus{\Ric_\mgt}_\mgt \leq \eta$
and $\mgt(\Rend \omega, \omega) \geq \zeta \norm{\omega}$
for $\omega \in \Forms(\cM)$.
That there exists $\kappa > 0$ such that
$\inj(\cM,\mgt) \geq \kappa$ also follows
from compactness of $\cM$ and smoothness of $\mgt$. See Theorem III.2.1 and
the discussion prior to Theorem III.2.3 in \cite{Chavel}.
The conclusion is then obtained by 
invoking Corollary \ref{Cor:KatoFn} and
Theorem \ref{Thm:KatoD}.
\end{proof}

\begin{remark}
If the metric $\mg$ was continuous,
then we can choose any
$C > 1$ and by invoking Proposition \ref{Prop:ContSmoothMet}, 
we can find $\mgt$ to be smooth and $C$-uniformly everywhere close 
to $\mg$.
\end{remark}
\section{Quadratic estimates and isometries}
\label{Sect:QuadIsom}

In our achievements so far, we have always considered
the situation of fixing a manifold and studying the persistence
of quadratic estimates under suitable changes of the metric.
Another important situation to consider
is the transmission of quadratic estimates
between manifolds which are isometric. 
Our ability to do this will depend on the 
regularity of the isometry. We first 
consider a general description of this problem
at the level of the AKM framework. 

\subsection{Isometries between Hilbert spaces}

The first results we obtain are 
concerned with pushing and pulling forward Dirac-type
operators on Hilbert spaces. 
Let $\Hil_1$ and $\Hil_2$
be two Hilbert spaces with 
inner products $\inprod{\mdot,\mdot}_1$
and $\inprod{\mdot,\mdot}_2$ respectively. 
We assume that $\Phi :\Hil_1 \to \Hil_2$
is an \emph{isometric isomorphism}
between $\Hil_1$ and $\Hil_2$, 
by which we mean that $\Phi$
is a vector space isomorphism satisfying
$\inprod{\Phi u, \Phi v}_2 = \inprod{u,v}_1$
for all $u,\ v \in \Hil_1$.
On letting $\Gamma_i: \Hil_i \to\Hil_i$
be closed, densely-defined operators
on $\Hil_i$ related via $\Phi$, we obtain 
the following transformation rule for their adjoints.
 
\begin{lemma}
\label{Lem:Oppush}
Suppose that $\Gamma_2 = \Phi  \Gamma_1 \Phi ^{-1}$,
by which we mean that $\dom(\Gamma_2) = \Phi \dom(\Gamma_1)$
and $\Gamma_2 u = \Phi  \Gamma_1 \Phi ^{-1} u$ 
for all $u \in \dom(\Gamma_2)$. Then, 
$\adj{\Gamma}_2 = \Phi \adj{\Gamma}_1 \Phi ^{-1}$
and $\close{\ran(\Gamma_2)} = \Phi  \close{\ran(\Gamma_1})$.
\end{lemma}
\begin{proof}
First, we prove that $\dom(\adj{\Gamma}_2) = \Phi \dom(\adj{\Gamma_1})$.
Fix, $u \in \dom(\adj{\Gamma}_2)$. Then, whenever $v \in \dom(\Gamma_2)$,
$$ 
\inprod{\adj{\Gamma}_2 u, v}_2 
	= \inprod{u, \Gamma_2 v}_2 
	= \inprod{\Phi ^{-1}u, \Phi ^{-1}\Gamma_2 v}_1
	= \inprod{\Phi ^{-1}u, \Gamma_1 \Phi ^{-1}v}_1.$$
But since $\dom(\Gamma_2) = \Phi \dom(\Gamma_1)$,
we can write $v' = \Phi ^{-1}v \in \dom(\Gamma_1)$
and 
furthermore, 
$\inprod{\adj{\Gamma}_2u,v}_2 = \inprod{\Phi ^{-1}\adj{\Gamma}_2u,\Phi ^{-1}v}_1$
and thus, 
$\inprod{\Phi^{-1}\adj{\Gamma}_2u, \tilde{v}}_1 = \inprod{ \Phi ^{-1}u, \adj{\Gamma}_1\tilde{v}}_1$
for all $\tilde{v} \in \dom(\Gamma_1)$. Therefore, 
$\Phi ^{-1}u \in\dom(\adj{\Gamma}_1)$.

For the opposite direction, let $u \in \dom(\adj{\Gamma}_1)$
and let $v \in \dom(\adj{\Gamma}_1)$. Then,
a similar calculation holds: 
$$
\inprod{\adj{\Gamma}_1u,v}_1 
	= \inprod{u, \Gamma_1 v}_1
	= \inprod{\Phi u, \Phi \Gamma_1 v}_2
	= \inprod{\Phi u, \Phi \Gamma_1 \Phi ^{-1} (\Phi v)}_2
	= \inprod{\Phi u, \Gamma_2(\Phi v)}_2.$$
So, 
$\inprod{\Phi \adj{\Gamma}_1u, \tilde{v}}_2 = \inprod{\Phi u, \Gamma_2\tilde{v}}$
for all $\tilde{v} \in \dom(\Gamma_2)$ and thus, we
conclude that $\Phi u \in \dom(\adj{\Gamma}_2)$.
Note that this implies that $\dom(\adj{\Gamma}_2) \subset \Phi \dom(\adj{\Gamma}_1)$,
for if not, then there would be a $u \in \dom(\adj{\Gamma}_2)$
such that $u \neq \Phi v$ for all $v \in \dom(\adj{\Gamma}_1)$,
but setting $v = \Phi^{-1}u \in \dom(\adj{\Gamma}_1)$ would produce a contradiction 
by what we have just proved. 
On combining these two calculations, 
we obtain that $\dom(\adj{\Gamma}_2) = \Phi \dom(\adj{\Gamma}_1)$.

Next, we fix $u \in \dom(\adj{\Gamma}_2)$ and compute: 
$$
\inprod{\adj{\Gamma}_2u, v}_2 
	= \inprod{\Phi ^{-1}u, \adj{\Gamma}_1 \Phi ^{-1} v}_1 
	= \inprod{ \adj{\Gamma}_1\Phi ^{-1}u, \Phi ^{-1} v}_1
	= \inprod{\Phi  \adj{\Gamma}_1 \Phi ^{-1}u, v}_2$$
for all $v \in \dom(\Gamma_2)$, and
so $\adj{\Gamma}_2u = \Phi \adj{\Gamma}_1 \Phi ^{-1} u$
by the density of  $\dom(\Gamma_2)$ in $\Hil_2$.

Now, let us show that $\close{\ran(\Gamma_2)} = \Phi  \close{\ran(\Gamma_1)}.$
Fix $u \in \dom(\Gamma_2)$. Then, write 
$v = \Gamma_2 u = \Phi  \Gamma_1 \Phi ^{-1}u$. That is
$ \Phi ^{-1}v = \Gamma_2 \Phi ^{-1}u$ which shows
that $\Phi ^{-1} \ran(\Gamma_2) \subset \ran(\Gamma_1)$.
For the other direction,
let $u \in \dom(\Gamma_1)$ and so 
$v = \Gamma_1 u = \Phi ^{-1}\Gamma_2 \Phi  u$.
Then, we can conclude by similar
reasoning that $\Phi \ran(\Gamma_1) \subset \ran(\Gamma_2)$
and hence $\ran(\Gamma_2) = \Phi \ran(\Gamma_1)$.
The proof is completed on observing that 
$\norm{v}_{1} = \norm{\Phi v}_{2}$.
\end{proof}

Let us now assume that the operator $\Gamma_1$
and $B_i \in \bddlf(\Hil_1)$ satisfy the 
hypotheses (H1)-(H3) of the AKM framework 
as outlined in 
\S\ref{Sect:AKM}.
By virtue of the previous lemma, the
definition of $\tilde{B}_i$, and as a consequence
of the fact that $\Phi$ is an 
isometry, the operators $\tilde{B}_i$
satisfy the same coercivity estimate in (H2)
with the exact same constants as $B_i$.
In fact, it is easy to see that 
the operators $\Gamma_2$ and $\tilde{B}_i$
satisfy the entire set of hypotheses (H1)-(H3).
Define $\Pi_B = \Gamma_1 + B_1 \adj{\Gamma}_1 B_2$
and $\tilde{\Pi}_B = \Gamma_2 + \tilde{B}_1 \Gamma_2 \tilde{B}_2$.
We note the following.

\begin{lemma}
$(\iden + t^2 \tilde{\Pi}_B^2)^{-1} u = \Phi (\iden + t^2 \Pi_B^2)^{-1}\Phi ^{-1}u$
for all $u \in \Hil_2$.
\end{lemma}
\begin{proof}
First, it is an easy fact that $\tilde{\Pi}_{B} = \Phi \Pi_B \Phi ^{-1}$.
So, now, fix $v = \Phi (\iden + t^2\Pi_B^2)^{-1}\Phi ^{-1}u$
so that $\Phi ^{-1}v = (\iden + t^2 \Pi_B^2)^{-1}\Phi ^{-1}u$.
Then,
$
\Phi ^{-1}u = (\iden + t^2\Pi_B^2)\Phi ^{-1}v = \Phi ^{-1}v + t^{2}\Pi_B^2\Phi ^{-1}v$
and multiplying both sides by $\Phi$ then yields
$u = (\iden + t^2 \tilde{\Pi}_B^2)v$.
Thus, $v = (\iden + t^2 \tilde{\Pi}_B^2)^{-1}u$.
\end{proof}

On combining these two lemmas, we obtain the following
result pertaining to the transmission of quadratic estimates
across isometries.

\begin{proposition}
\label{Prop:QuadEstT}
The quadratic estimate
$$
\int_{0}^\infty \norm{t\tilde{\Pi}_B(\iden + t^2\tilde{\Pi}_B^2)^{-1}u}^2_2\ \frac{dt}{t}
\simeq \norm{u}_2^2$$
for $u \in \close{\ran(\tilde{\Pi}_B)}$
is satisfied if and only if
$$ 
\int_{0}^\infty \norm{t\Pi_B(\iden + t^2 {\Pi}_B^2)^{-1}v}^2_1\ \frac{dt}{t}
\simeq \norm{v}_1^2$$
for all $v \in \close{\ran(\Pi_B)}$.
\end{proposition}
\begin{proof}
First, by identifying $\Pi_B$ with $\Gamma_1$
in Lemma \ref{Lem:Oppush}, we find that
$\close{\ran(\tilde{\Pi}_B)} = \Phi \close{\ran(\Pi_B)}$.
Thus, for $u \in \close{\ran(\tilde{\Pi}_B)}$
we have that $\norm{u}_2 = \norm{\Phi ^{-1}u}_1$.
For the same $u$, 
$$
\norm{t\tilde{\Pi}_B(\iden + t^2\tilde{\Pi}_B^2)^{-1}u}_2
	= \norm{t\Phi ^{-1} \Pi_B \Phi ^{-1}\Phi (\iden + t^2 \Pi_B^2)^{-1}\Phi ^{-1}u}_2
	= \norm{t\Pi_B(\iden + t^2 \Pi_B^2)^{-1}\Phi ^{-1}u}_1.$$
Setting $v = \Phi ^{-1}u$ finishes the proof.
\end{proof}

\subsection{Local Lipeomorphisms between manifolds}

Let $\cM$ and $\cN$ be smooth manifolds and $\mh$
a $\Ck{k}$ ($k \geq 1$) metric on $\cN$.
We would like to consider maps $F:\cM \to \cN$
with which we can pull the metric $\mh$ across 
to $\cM$ in a way that this new geometry
reflects the regularity of $F$. 
If $F$ is a $\Ck{k}$ ($k\geq 1$) diffeomorphism, then 
we are able to simply consider the pullback 
metric of $\mh$ of regularity
$\Ck{k-1}$.

The key point to notice here is that we require
the first derivatives of $F$ to exist on a suitably
large set of points in $\cM$.  
At a glance, it seems that it would suffice
to ask $F$ to be a  \emph{Lipeomorphism},
which we define as an invertible Lipschitz map
with a Lipschitz inverse between two metric spaces. 
However, without specifying a metric a priori
on $\cM$, we are unable to make sense of this terminology. 
The existence of derivatives
is a local problem and therefore the notion of a
\emph{local Lipeomorphism} can be formulated
between two manifolds by resorting to their locally
Euclidean structure. 
We begin our discussion by formalising this notion. 
We note that the definition given below
is independent of the metric $\mh$ on $\cN$.

\begin{definition}[Local Lipeomorphism]
Let $\cM$ and $\cN$ be two smooth manifolds. 
Then, we say that 
$F: \cM \to \cN$ is \emph{local Lipeomorphism} if
\begin{enumerate}[(i)]
\item $f$ is a homeomorphism, and
\item for all $x \in \cM$, there exists a chart
	$(U, \psi)$ near $x$ and $(V, \phi)$ near 
	$F(x)$ and a constant $C \geq 1$ such that
	$$C^{-1} \modulus{x' - y'} \leq 
		\modulus{(\phi \comp F \comp \psi^{-1})(x') - (\phi \comp F \comp \psi^{-1})(y')} \leq
		C \modulus{x' - y'}.$$
\end{enumerate}
For the sake of nomenclature, we call the charts $(U, \psi)$ and $(V, \phi)$
the \emph{Lipeo-admissible} charts.
\end{definition} 

As a consequence of this definition, 
the differential of $\tilde{F} = (\phi \comp F \comp \psi^{-1})$ exists
almost-everywhere on Lipeo-admissible charts
and hence,  we are able to define the pushforward $\pushf{F}$ 
almost-everywhere in $\cM$. 
 
From this point onwards, let us assume
$\mh$ to be a \emph{complete} $\Ck{k}$ ($k \geq 0$) metric on $\cN$.
Define $\mg = \pullb{F}\mh$ and note that it is a 
\emph{rough metric}. 
We will prove in this section
that $\mg$ has considerably better properties
than an arbitrary rough metric.

\subsubsection{The distance metric $\met_\mg$ and geodesy}

Recall that the length of an 
absolutely continuous curve $\gamma: I \to \cM$ 
is given by:
$$
\len_\mg(\gamma) 
	= \int_{I} \modulus{\gamma(t)}_{\mg(\gamma(t))}\ dt
	= \int_{I} \modulus{F \comp \gamma(t)}_{\mh(F \comp \gamma(t))}\ dt < \infty.$$
As for continuous metrics, 
by taking an infimum over the lengths $\len_\mg(\gamma)$
for such curves $\gamma$ between points $x, y\in \cM$,
we obtain a distance metric $\met_\mg$. 
The following proposition then gives a regularity criteria
for geodesics of $(\cM, \mg)$.

\begin{proposition}
For every $x,y \in \cM$, $\met_\mg(x,y) = \met_\mh(F(x),F(y))$. 
Furthermore, if $x', y' \in \cN$ with a $\Ck{k}$-minimising geodesic (for $k \geq 1$)
between them, then there exists a Lipschitz curve
on $\cM$ that is a minimising geodesic between $F^{-1}(x')$ and $F^{-1}(y')$.
\end{proposition}
\begin{proof}
Fix $x, y \in \cM$ and let
$\gamma: [0,1] \to \cN$ be an absolutely continuous 
curve between $F(x)$ and $F(y)$. Then, 
$\met_{\mg}(x,y) \leq \len_\mg(F^{-1} \comp \gamma) = \len_\mh(\gamma).$
Taking an infimum over such curves then gives that
$\met_\mg(x,y) \leq \met_\mh(F(x),F(y))$.
Conversely, if $\sigma:[0,1] \to \cM$ is an
absolutely continuous curve between $x$ and $y$, then
$\met_\mh(F(x),F(y)) \leq \len_\mh(F \comp \sigma) = \len_\mg(\sigma)$
and so $\met_\mh(F(x),F(y)) \leq \met_\mg(x,y)$.

Suppose now that $\gamma:[0,1] \to \cN$ is a minimising geodesic
between $x', y' \in \cN$ that is of class $\Ck{k}$ for
$k \geq 1$. Then,  for $t_1, t_2 \in I$,
$\met_\mh(\gamma(t_1), \gamma(t_2)) = \modulus{t_1 - t_2}$.
On letting $\sigma = F \comp \gamma$, 
we note that 
$\met_\mg(\sigma(t_1), \sigma(t_2)) = \met_\mh(\gamma(t_1), \gamma(t_2))$, 
which shows that $\sigma$
is a minimising geodesic in $\mg$. It is easy to see that 
$\sigma$ is Lipschitz. 
\end{proof}

\begin{remark}
\begin{enumerate}
\item Burtscher points out (in a private communication)
	that for $\Ck{1}$ metrics there exist 
	``geodesics'' in the sense of curves
	satisfying the Euler-Lagrange equations
	that are nowhere minimising due to 
	a result of Hartman-Wintner in \cite{HW}.
	More seriously such curves may not be 
	unique (see \cite{H}).
	This forces us to only consider
	minimising geodesics.
	
\item 	When a  metric is $\Ck{0,1}$,
	it turns out that the regularity of its minimising 
	geodesics (when they exist) are 
	$\Ck{1,1}$. When it is $\Ck{0,\alpha}$ for $\alpha \in (0,1)$, 
	then its minimising geodesics are $\Ck{1, \alpha/2}$. 
	Indeed, minimising geodesics of
	a $\Ck{k}$ metric ($k \geq 1$)
	are $\Ck{k+1}$.
	In each case, we see that the minimising geodesics have 
	one integer 
	exponent of regularity higher than the metric.
	While this may tempt us to embrace this observation as a maxim, 
	it is not true for purely continuous metrics!
	Pettersson in his remarkable paper \cite{Pettersson}
	gives a continuous metric on $\R^2$ for which a
	minimising geodesic passing through the origin is given by 
	$t(\sin a(t), \cos a(t))$ where $a(r) = \log( - \log\modulus{r})$.
	It is easy to see that this curve is continuous
	but it is not differentiable at
	the origin.
\end{enumerate}
\end{remark}

\subsubsection{The induced measure $\mu_\mg$}

Recall the induced measure for rough metrics 
from \S\ref{Sect:RoughM}. Here,
we establish a relationship between
the measures $\mu_\mh$ and $\mu_\mg$.
This is necessary in order for us to relate the
Lebesgue and Sobolev space theory of the two geometries.  
First, we present the following lemma which 
provides us with a formula expressing $\mu_\mg$
with respect to $\mu_\mh$ inside charts.

\begin{lemma}
\label{Lem:MeasureRep}
Let $(U, \psi)$ be a chart on $\cM$ 
near $x \in \cM$ and $(V, \phi)$ a chart 
near $F(x) \in \cN$. Then,
$$ d\mu_\mg(y)  = \modulus{\det D\tilde{F}(y)}\ \sqrt{\det \mh(\tilde{F}(y))}\ d\Leb(y),$$
where $\tilde{F} = \phi \comp F \comp \psi^{-1}$
and for $\Leb$-almost every $y \in \psi(U)$.
\end{lemma}
\begin{proof}
Let $\set{x^i}$ denote coordinates in $U$
and $\set{y^j}$ coordinates in $V$. 
It suffices to show that 
$\det \mg(y) = \modulus{\det \Dir\tilde{F}(y)}^2\ \det \mh(\tilde{F}(y))$
for $\Leb$-a.e. $y \in U$.

Inside the chart $V$, we can write
$\mh(w) = \mh_{kl}(w) \ dy^k \tensor dy^l$ and hence, 
by the linearity of the pullback,
$\pullb{F} \mh(y) = \mh_{kl}(F(y))\ \pullb{F}dy^k \tensor \pullb{F}dy^l.$
A straightforward calculation gives
that $\pullb{F}dy^k = \partial_{x^i} \tilde{F} dx^i$
and hence,
$\pullb{F}\mh(y) = \mh_{kl} \partial_{x^i}\tilde{F}^k \partial_j \tilde{F}^l\ dx^i \tensor dx^j.$
Therefore, $\mg_{ij}(y) = \mh_{kl}(F(y) \partial_{x^i}\tilde{F}^k(y) \partial_j\tilde{F}^l(y)$
and so $(\mg_{ij}) = (\mh_{kl})\Dir\tilde{F}\mdot\Dir\tilde{F}$.
Thus, $\det \mg(y) = (\det\mh(F(y))) (\det \Dir\tilde{F})^2$. 
\end{proof}

On combining this result with Lemma \ref{Lem:WeightedMeas},
we are able to compare the measure algebras 
of $\mu_\mg$ and $\mu_\mh$ via $F$. 
 
\begin{proposition}
A function $\xi: \cM \to \C$ is $\mu_\mg$-measurable
if and only if $\pullb{F}{\xi} = \xi \comp F^{-1}: \cN \to \C$
is $\mu_\mh$-measurable.  
\end{proposition}
\begin{proof}
Let $(U_i,\psi_i)$ and correspondingly $(V_j = F(U_j), \phi_j)$
be Lipeo-admissible charts and fix $\alpha \in \R$. 
Define
\begin{align*}
X_j &= U_j \intersect \xi^{-1}(\alpha, \infty] = \set{x \in U_j: \xi(x) > \alpha} \\
Y_j &= V_j \intersect (\xi \comp F^{-1})^{-1}(\alpha, \infty] 
	= \set{y \in V_j: \xi \comp F^{-1}(x) > \alpha}. 
\end{align*}
Then, it is easy to see that $x \in X_j$ 
if and only if $F(x) \in V_j = F(U_j)$
and $\xi(x) > \alpha$ if and only if $(\xi \comp F^{-1})(F(x)) > \alpha$.
Thus, $Y_j = F(X_j)$.

By Lemma \ref{Lem:MeasureRep}, we have inside $U_j$
that $d\mu_\mg(x) = \modulus{\det \Dir \tilde{F}(x)} \sqrt{\det\mh(F(x))}\ d\Leb(x)$
and therefore, on setting $f = \modulus{\det \Dir \tilde{F}(x)} \sqrt{\det\mh(F(x))}$,
we can conclude from Lemma \ref{Lem:WeightedMeas}
that $X_j$ is $\mu_\mg$-measurable if and only if 
$Y_j$ is $\mu_\mh$-measurable. Thus,
$\xi^{-1}(\alpha, \infty] = \union_j X_j$ is $\mu_\mg$
measurable if and only if $(\xi \comp F^{-1})^{-1}(\alpha, \infty] = \union_j Y_j$
is $\mu_\mh$-measurable.
\end{proof} 

As we would expect, we obtain that $\mu_\mg$ is 
indeed the pullback measure of $\mu_\mh$ under $F$. 

\begin{proposition}
\label{Prop:MeasPullb}
The measure $\mu_\mg = \pullb{F}\mu_\mh$. That is
if $\xi: \cM \to \C$ $\mu_\mg$-integrable, then,
$$\int_\cM \xi(x)\ \mu_\mg(x) = \int_{\cN} \xi \comp F^{-1}(y)\ d\mu_\mh(y).$$ 
\end{proposition}
\begin{proof}
Fix $(U,\psi)$ and $(V,\phi)$ Lipeo-admissible
charts with $V = F(U)$. Then, first consider $\xi: \cM \to \C$
with $\spt \xi \subset U$ which is $\mu_\mg$
integrable.
Then,
$$ \int_{\cM} \xi(x)\ d\mu_\mg(x) 
	= \int_{\psi(U)} \xi \comp \psi^{-1}(x)\ \sqrt{\det \mg(x)}\ d\Leb(x).$$
Also, $\tilde{\xi} = \xi \comp F^{-1}$
is $\mu_\mh$ measurable 
with $\spt \tilde{\xi} \subset V$ and 
$$ 
\int_{\cN} \tilde{\xi}(y)\ d\mu_{\mh}(y)
	= \int_{\phi(V)} \xi \comp F^{-1} \comp \phi^{-1}(y) \sqrt{\det \mh(y)}\ d\Leb(y).$$

Now, by our choice of $(U, \psi)$ and $(V, \phi)$, 
we have that the map 
$\tilde{F} = \phi \comp F \comp \psi^{-1}: \psi(U) \to \phi(V)$
is bi-Lipschitz and therefore
we have an integration by substitution formula: 
\begin{multline*}
\int_{\psi(U)} \xi \comp \psi^{-1}(x)\ \sqrt{\det \mg(x)}\ d\Leb(x) \\
	= \int_{\tilde{F}(\psi(U))} \xi \comp \psi^{-1} \comp \tilde{F}^{-1}(x)
		\sqrt{\det\mg(\tilde{F}^{-1}(x))} \modulus{\det \Dir \tilde{F}^{-1}(x)}\ d\Leb(x).
\end{multline*}
From Lemma \ref{Lem:MeasureRep},
$\sqrt{\det\mh}(x) = \sqrt{\det\mg(\tilde{F}^{-1}(x))} \modulus{\det \Dir \tilde{F}^{-1}(x)}$
and 
$\xi^\comp \psi^{-1} \comp \tilde{F}^{-1}(x) = \tilde{\xi} \comp \phi^{-1}(x)$.
Therefore, $\int_{U} \xi\ d\mu_\mg = \int_{V} \xi \comp F^{-1}\ d\mu_\mh$.

For a general $\xi$, we can cover $\cM$ by $(U_i, \psi_i)$
and corresponding $(V_i, \phi_j)$ and the patch the
integral together through a partition of unity.
\end{proof}

As a direct consequence, we point out the following
regularity result.

\begin{proposition}
$\mu_\mg$ is a Radon measure.
\end{proposition}
\begin{proof}
Since $\mg$ is a rough metric, by Proposition \ref{Prop:BorelCpct}, 
we obtain that $\mu_\mg$ is Borel and finite on compact sets.
Hence, we only need to prove Borel-regularity. For that, 
let $A \subset \cM$. 
Then, there exists $\tilde{B} \subset \cN$ such that
$F(A) \subset \tilde{B}$ and $\mu_\mh(F(A)) = \mu_\mh(\tilde{B})$. 
Set $B = F^{-1}(\tilde{B})$ and note that $A \subset B$ and
$\mu_\mg(A) = \mu_\mh(F(A)) = \mu_\mh(\tilde{B}) = \mu_\mg(B)$.
\end{proof} 

\subsubsection{Lebesgue and Sobolev spaces}

In our previous analysis where the
manifold $\cM$ was fixed and
we dealt with two uniformly close metrics $\mg$ and $\mgt$,
we were able to relate the
Lebesgue and Sobolev spaces of the two metrics to each
other in a more or less straight forward manner. 
The situation we now face is different -
we need to relate these spaces via $F$. The primary difficulty is that
the pullback of $F$ does not preserve smoothness.
Rather, it sends smooth functions to Lipschitz ones, 
and smooth tensors to tensors with only measurable
coefficients.
As a consequence we we dispense with
our attempts to setup this analysis on differential forms
and only concentrate on functions.
%\Bl
%The techniques from \cite{GMM} seems hopeful 
%to conduct such an analysis,
%but we still choose to abstain from
%pursuing such a line of inquiry
%as it would detract from the point 
%of this paper and perhaps bedevil the reader
%with a morass of technicalities. \Bk
Furthermore, we demonstrate the somewhat unsurprising
fact that Lipschitz functions play
a sufficient role in the Lebesgue and Sobolev
theory so that we may use them in our analysis
instead of smooth objects.

We begin our efforts by presenting the following
easy but important observation which is
immediate from Proposition \ref{Prop:MeasPullb}.

\begin{proposition}
The map $F$ induces an isometry between 
$\Lp{p}(\Tensors[a,b] \cM,\mg)$ and $\Lp{p}(\Tensors[a,b]\cN,\mh)$.
\end{proposition}
%\begin{proof}
%It is immediate that
%$$
%\int_{\cM} \modulus{T}_{\mg}^p\ d\mu_\mg = \int_{\cN} \modulus{\pullb{F^{-1}}T}_\mh^p\ d\mu_\mh$$
%by Proposition \ref{Prop:MeasPullb}, which proves this.
%\end{proof}

Let us now fix some notation.
Fix $\cP$ to be a smooth manifold and 
let us denote the exterior derivative on this manifold
by $\conn^\cP$.
Let $\Lips[loc](\cP)$ be the space of locally Lipschitz functions, 
defined by appealing to the local Euclidean structure
of $\cP$
so this space can be formulated independently of a metric on $\cP$.
Define the subspace of such functions with compact support 
by $\Lips[c](\cP)$. 
The idea is to substitute $\Lips[loc](\cP)$
for $\Ck{k}(\cP)$. The following lemma gives credence
to our efforts.

\begin{lemma}
For every $f \in \Lips[loc](\cN)$, $\conn^\cN f$ exists for $\mu_\mh$-a.e.
Furthermore, if $\xi \in \Ck{k}(\cM)$ for $k \geq 1$, then 
$\pullb{F^{-1}}\xi \in \Lips[loc](\cN)$ and $\conn^\cM\xi = \pullb{F}\conn^{\cN}\pullb{F^{-1}} \xi$
$\mu_\mg$-a.e.
Similarly, if $\eta \in \Ck{k}(\cN)$ for $k \geq 1$, then, 
$\pullb{F}\eta \in \Lips[loc](\cM)$ and 
$\conn^\cN \eta = \pullb{F^{-1}}\conn^\cM \pullb{F}\eta$ 
$\mu_\mh$-a.e.
\end{lemma}
\begin{proof}
The conclusions follows immediately from the
fact that $\conn^\cP$ is the exterior derivative
on $\cP$, $\pullb{\psi}\conn^\cP = \conn^{\R^n} \pullb{\psi}$,
and since $\phi \comp F \comp \psi^{-1}$ is Lipschitz and a.e. 
differentiable.
\end{proof}  

The following proposition 
then yields that $\Lips[c](\cN)$
is a suitable substitute for $\Ck[c]{k}(\cN)$, $k \geq 1$.

\begin{proposition}
$\Lips[c](\cN) \subset \SobH{1}(\cN)$ and
moreover, $\Lips[c](\cN)$ is dense in $\SobH{1}(\cN)$.
\end{proposition}
\begin{proof}
First, consider $f \in \Lips[c](\cN)$
with $\spt f \subset V$, where $(V, \phi)$
is a compact chart on $\cN$. Then,
$f \comp \phi^{-1}: \phi(V) \to \C$ is
Lipschitz, and furthermore,
$\spt (f \comp \phi^{-1}) \subset \phi(V)$.
Then, there exists $\delta > 0$ such that for
all $\epsilon < \delta$,
$\spt ( \eta^\epsilon \convolve f \comp \phi^{-1}) \subset \phi(V)$, 
where $\eta$ is the standard symmetric mollifier.
Writing
$$ \tilde{f}_\epsilon 
	= \begin{cases}
		(\eta^\epsilon \convolve f \comp \phi^{-1}) \comp \phi(x)	&x \in V\\
		0 &x \not\in V,
	\end{cases} $$
defines $\tilde{f}_\epsilon \in \Ck[c]{\infty}(\cN)$.
Then, 
\begin{align*} 
\int_{\cN} \modulus{\tilde{f}_\epsilon - f}^2\ d\mu_\mh
	&= \int_{V} \modulus{\tilde{f}_\epsilon - f}^2\ d\mu_\mh \\
	&= \int_{\phi(V)} \modulus{\eta^\epsilon \convolve f \comp \phi^{-1} - f \comp \phi^{-1}}^2\ 
		\sqrt{\det \mh}\ d\Leb \\
	&\leq C \int_{\phi(V)} \modulus{\eta^\epsilon \convolve f \comp \phi^{-1} - f \comp \phi^{-1}}^2\ d\Leb
	\to 0
\end{align*}
as $\epsilon \to 0$ and where $C$ depends on $\mh$ and $V$ by
the continuity of $\mh$ and compactness of $V$.

Also,  
\begin{multline*}
\int_{\cN} \modulus{\extd \tilde{f}_\epsilon - \extd \tilde{f}_{\epsilon'}}_\mh^2
	= \int_{\phi(V)} \modulus{ \pullb{\phi^{-1}} \extd \tilde{f}_\epsilon - 
		\pullb{\phi^{-1}} \extd \tilde{f}_{\epsilon'}}_\mh^2\ \sqrt{\det \mh}\ d\Leb \\
	\leq C \int_{\phi(V)} \modulus{ \pullb{\phi^{-1}} \extd \tilde{f}_\epsilon - 
		\pullb{\phi^{-1}} \extd \tilde{f}_{\epsilon'}}^2\ d\Leb
\end{multline*}
where $C$ depends on $\mh$ and $V$ as before. 
Also, 
$$\partial_i \eta^\epsilon \convolve f \comp \phi^{-1} 
	= \eta^\epsilon \convolve \partial_i(f \comp \phi^{-1})$$
and therefore,
$$\int_{\phi(V)} \modulus{ \pullb{\phi^{-1}} \extd \tilde{f}_\epsilon - 
		\pullb{\phi^{-1}} \extd \tilde{f}_{\epsilon'}}_\mh^2\ d\Leb \to 0$$
as $\epsilon, \epsilon' \to 0$, 
which implies that there exists $v \in \Lp{2}(\cN,\mh)$
such that $\extd \tilde{f}_\epsilon = \conn^\cN \tilde{f}_\epsilon \to v$.
Combining this with the fact that $\tilde{f}_\epsilon \to f$
and because $\conn^\cN$ is a closable operator,
we have that $f \in \SobH{1}(\cN)$ and $v = \close{\conn^\cN} f$.
Through a partition of unity argument,
the requirement that $V$ is compact and $\spt f \subset V$
can be dropped and makes the result valid for
every $f \in \Lips[c](\cN)$. 
\end{proof}

Recall that, since $\mg$ is a rough metric,
the two operators $\conn[c]^\cM$ and $\conn[2]^\cM$
are both automatically densely-defined are closable
as a consequence of Proposition \ref{Prop:OpSobRough}.
Since $\close{\conn[c]^\cN} = \close{\conn[2]^\cN}$
due the completeness of $\mh$, we 
obtain the following 
similar result on $\cM$.

\begin{proposition}
$\SobH[0]{1}(\cM) = \SobH{1}(\cM)$ and 
$\pullb{F}\SobH{1}(\cN) = \SobH{1}(\cM)$
with $\close{\conn[2]^\cM} = \pullb{F}\close{\conn[2]^\cN}\pullb{F^{-1}}$.
\end{proposition}
\begin{proof}
First, we show that 
$\pullb{F}\SobH{1}(\cN) = \SobH{1}(\cM)$.
For this, fix $f \in \SobH{1}(\cM)$. Then,
there exists $f_j \in \Ck{\infty} \intersect \Lp{2}(\cM) \to \Lp{2}(\cotanb\cM)$
such that $f_j \to f$ and $\conn^\cM f_j \to \close{\conn[2]^\cM}f$.
But $\pullb{F^{-1}} f_j \to \pullb{F^{-1}}f$
and $\conn^\cN \pullb{F^{-1}} \to v$. By the closedness
of $\close{\conn[2]^\cN}$, we have that $\pullb{F^{-1}f} \in \SobH{1}(\cN)$
and $v = \close{\conn[2]^\cN}\pullb{F^{-1}}f$.
This shows that $\pullb{F^{-1}} \SobH{1}(\cM) \subset \SobH{1}(\cN)$. 
A similar argument then establishes that
$\pullb{F}\SobH{1}(\cN) \subset \SobH{1}(\cM)$.
The formula $\close{\conn[2]^\cM} = \pullb{F}\close{\conn[2]^\cN}\pullb{F^{-1}}$
follows immediately. 

Next we prove $\SobH[0]{1}(\cM) = \SobH{1}(\cM)$.
For this, note that since $\pullb{F}\SobH{1}(\cN) = \SobH{1}(\cM)$
and $\pullb{F}$ is an $\Lp{2}$ isometry 
implies that $\pullb{F}\Lips[c](\cN)$ is a dense 
subset of $\SobH{1}(\cM)$. It is enough to prove that
$\pullb{F}\Lips[c](\cN) \subset \SobH[0]{1}(\cM)$.
In fact, we prove that $\Lips[c](\cM) \subset \SobH[0]{1}(\cM)$
is a dense subset, and it is
an easy observation that $\pullb{F}\Lips[c](\cN) \subset \Lips[c](\cM)$.

Let $f \in \Lips[c](\cM)$ such that
$\spt f  \subset U$, where $(U, \psi)$ and $(F(U) \subset V,\phi)$ is
a Lipeo-admissible compact chart. Then, 
we can arrange $\epsilon > 0$ small such that
$f_\epsilon = \pullb{\psi} \comp (\eta^\epsilon \convolve f \comp \psi^{-1}) \in \Ck[c]{\infty}(U)$, 
and extend it to $0$ outside of $U$. Then, 
$f_\epsilon \to f$ in $\Lp{2}(\cM)$
and 
\begin{align*}
\int_{\cM} &\modulus{\conn[c]^\cM f_\epsilon - \conn[c]^\cM f_{\epsilon'}}\ d\mu_\mg \\
	 &= \int_{\psi(U)} \modulus{\extd (\eta^\epsilon \convolve f \comp \psi^{-1}) 
	- \extd (\eta^\epsilon \convolve f \comp \psi^{-1})}\ 
	\modulus{\det \Dir \tilde{F}} \sqrt{\det \mh}(\tilde{F}) \ d\Leb \\
	 &\leq C \int_{\psi(U)} \modulus{\extd (\eta^\epsilon \convolve f \comp \psi^{-1}) 
	- \extd (\eta^\epsilon \convolve f \comp \psi^{-1})}\ d\Leb \to 0
\end{align*}
where $\tilde{F} = \phi \comp F \comp \psi^{-1}$ and
 the $C$ depends on the lower bound of 
$\modulus{\det \Dir \phi \comp F \comp \psi^{-1}}$, 
which exists since $U$ and $V$ are Lipeo-admissible
and by the continuity of $\sqrt{\det \mh}$ in $(\phi \comp F)(U)$
which is guaranteed to have compact closure.
Therefore, $\conn[c]^\cM f_\epsilon \to v$ as
$\epsilon \to 0$ and by the closedness
of $\close{\conn[c]^\cM}$, we conclude $f \in \SobH[0]{1}(\cM)$.
\end{proof}

\subsubsection{Transmitting quadratic estimates via $F$}

We now demonstrate how to
transmit quadratic estimates via 
$F$ between $\Lp{2}(\cM)$ and $\Lp{2}(\cN)$.
As we have previously mentioned,
we only concentrate on the case of
functions.

As before, let $S_\cM = (I, \close{\conn_c^\cM})$
with domain $\dom(S_\cM) = \SobH{1}(\cM)$
and similarly define $S_\cN$.
Then, write $\Hil_1 =
\Lp{2}(\cM) \oplus \Lp{2}(\cotanb\cM)	
\oplus \Lp{2}(\cotanb\cM)$
and define $\Gamma_\cM: \Hil_1 \to \Hil_1$ 
by 
$$\Gamma_\cM = \begin{pmatrix}0 & 0 \\ S_\cM & 0\end{pmatrix}.$$
Similarly, define $\Gamma_\cN$ upon replacing $S_\cM$ by $S_\cN$.
Since $\Hil_1$ is a Hilbert space and $\Gamma_\cM$ is densely-defined
and closed, it follows that $\adj{\Gamma}_\cM$ exists and it
is also densely-defined and closed.

Write $\Hil_2 = \Lp{2}(\cN) \oplus \Lp{2}(\cotanb\cN) \oplus \Lp{2}(\cotanb\cN)$
and define $\Phi:\Hil_1 \to \Hil_2$
by 
$$\Phi(u,v,w) = (\pullb{F^{-1}}u, \pullb{F^{-1}}v,\pullb{F^{-1}}w),$$
which is readily checked to be an isometry between $\Hil_1$ and $\Hil_2$.
Letting $\divv_\mg$ denote the divergence with respect to $\conn^\cM$
and metric $\mg$, we note that 
$\adj{\Gamma}_\cM = \Phi^{-1}\adj{\Gamma}_\cN\Phi$ 
as a consequence of Lemma \ref{Lem:Oppush}. That is precisely 
$\divv_\mg = \pullb{F} \divv_\mh \pullb{F^{-1}}$.

Let $A \in \Lp{\infty}(\bddlf(\cM \oplus \cotanb\cM))$, 
$a \in \Lp{\infty}(\bddlf(\cM)$
and suppose they satisfy
$$
\re\inprod{AS_\cM v, v}_1 \geq \kappa_1 \norm{v}_{\SobH{1}(\cM)}
\quad\text{and}\quad
\re\inprod{au, u}_1 \geq\kappa_2 \norm{u}_1$$
for $v \in \SobH{1}(\cM)$ and $u \in \Lp{2}(\cM)$.
Set
$$B_1 = \begin{pmatrix}a & 0 \\ 0 & 0 \end{pmatrix}
\quad\text{and}\quad
B_2 = \begin{pmatrix} 0 & 0 \\ 0 & A\end{pmatrix}.$$
Write $\tilde{B}_i = \Phi^{-1}B_i\Phi$, 
so that 
$$
\tilde{B}_1 = \begin{pmatrix} 
	\pullb{F} a \pullb{F^{-1}} & 0 \\
	0 & 0 \end{pmatrix}
\quad\text{and}\quad
\tilde{B}_2 = \begin{pmatrix}
	0 & 0 \\ 0 & \pullb{F}A \pullb{F^{-1}}
	\end{pmatrix}.$$
Then, on letting 
$\Pi_{\cM,B} = \Gamma_{\cM} + B_1 \adj{\Gamma}_\cM B_2$
and $\Pi_{\cN,\tilde{B}} = \Gamma_\cN + \tilde{B}_1 \adj{\Gamma}_\cN \tilde{B}_2$, 
we obtain the following main theorem of this section.

\begin{theorem}
\label{Thm:QuadEstLipeo}
The quadratic estimate
$$\int_0^\infty \norm{t\Pi_{\cM,B}(\iden + t^2 \Pi_{\cM,B}^2)^{-1}u}^2_\mg\ \frac{dt}{t} \simeq \norm{u}_\mg^2$$
for all $u \in \close{\ran(\Pi_{\cM,B})}$ 
if and only if 
$$\int_0^\infty \norm{t\Pi_{\cN, \tilde{B}}(\iden + t^2 \Pi_{\cN, \tilde{B}}^2)^{-1}v}^2_\mh\ \frac{dt}{t} \simeq \norm{v}_\mh^2$$
for all $v \in \close{\ran(\Pi_{\cN,\tilde{B}})}$.
\end{theorem}

\subsection{Lipschitz transformations}

Let $(\cM,\mg)$ and $(\cN,\mh)$ now be two smooth manifold with 
continuous metrics. 
Suppose that the map $F: (\cM, \mg) \to (\cN, \mh)$
is now a \emph{Lipeomorphism},
 by which we mean that
there exists $C \geq 1$ satisfying
$C^{-1} \met_\mg(x,y) \leq \met_\mh(F(x),F(y)) \leq C \met_\mg(x,y)$.
Our first observation is 
that such a map is a local Lipeomorphism as we
have previous defined.

\begin{proposition}
$F:\cM \to \cN$ is a local Lipeomorphism.
\end{proposition}
\begin{proof}
Since $F$ is an invertible Lipschitz map with
a Lipschitz inverse and since the Lipschitz property 
implies continuity, we have that
$F$ is a homeomorphism in the topologies induced on $\cM$ and $\cN$
by $\mg$ and $\mh$ respectively. But, by the continuity of $\mg$ and $\mh$, 
the induced topologies agrees with the natural topologies of the manifolds.

Now, let us show that it is locally Lipschitz. 
Choose charts $(U, \psi)$ near $x$ and
$(V,\phi)$ near $F(x)$ such that
$\close{U},\close{V}$ are compact
and that $V = F(U)$. 
Let $\tilde{F} = \phi \comp F \comp \psi^{-1}$.
Then, for any two points $x',y' \in \phi(U)$
and any absolutely continuous curve $\gamma$
connecting these points. Then,
$$ \modulus{\dot{\gamma}(t)}^2 
	= \mh(\dot{\gamma}(t), \dot{\gamma}(t))
	= \mh_{ij} \dot{\gamma}^i(t) \dot{\gamma}^j(t)
	\leq \cbrac{\sup_{x \in U} \max_{i} \mh_{ij}(x)} \modulus{\dot{\gamma}}_{\delta}(t)^2,$$
where continuity of $\mh$ and compactness of $U$ guarantees
the finiteness of the supremum.
Therefore, $\met_\mg(x',y') \lesssim \modulus{x' - y'}$ and therefore, 
$\met_\mh(\tilde{F}(x'), \tilde{F}(y')) \lesssim \modulus{x' - y'}$. But by considering
an absolutely continuous curve $\sigma$ connecting 
the points $\tilde{F}(x')$ and $\tilde{F}(y')$ in $\phi(V)$, through a similar argument
we obtain that
$\modulus{\tilde{F}(x') - \tilde{F}(y')} \lesssim \met_\mh(\tilde{F}(x'), \tilde{F}(y')),$
with the constant depending on $\mh$ and $V$.
In fact, a repetition of argument yields
$\met_\mh(\tilde{F}(x'), \tilde{F}(y')) \lesssim \modulus{\tilde{F}(x'), \tilde{F}(y')}$ and
since $\met_\mg(x',y') \lesssim \met_\mh(\tilde{F}(x'),\tilde{F}(y'))$,
by again repeating this argument, we get that
$\modulus{x' - y'} \lesssim \met_\mg(x',y')$.
This shows that $\tilde{F}$ is a bi-Lipschitz
mapping between $\psi(U)$ and $\phi(V)$, 
with the constant depending on $U, V$, $\mh$
and the global Lipschitz constant $C$. 
\end{proof}

Let us now apply this result to 
the case $(\cN,\mh) = (\cM, \mg)$.
Consider the pullback geometry $\mgt = \pullb{F}\mg$. 
This geometry is a \emph{Lipschitz transformation}
of $(\cM, \mg)$. 
As a consequence of Theorem \ref{Thm:QuadEstLipeo}, 
we are able to conclude that if
quadratic estimates are satisfied for
$(\cM, \mg)$, then they are also satisfied for
$(\cM, \mgt)$. 

\begin{theorem}
\label{Thm:KatoLipInv}
Let $(\cM,\mg)$ be a smooth manifold with continuous metric. 
Then the quadratic estimate
$$\int_0^\infty \norm{t\Pi_{\cM,B}(\iden + t^2 \Pi_{\cM,B}^2)^{-1}u}^2\ \frac{dt}{t} \simeq \norm{u}^2$$
for all $u \in \close{\ran(\Pi_{\cM,B})}$
remains invariant under Lipschitz transformations 
of $(\cM, \mg)$.
\end{theorem} 
\section{Metrics without lower bounds on injectivity radius}
\label{Sect:Inj}

In each of the papers \cite{AKMc}, \cite{AKM2}, \cite{Morris3}, \cite{BEMc}
and \cite{BMc} which establish solutions to the Kato square root
problem, the underlying geometry possesses lower bounds on injectivity 
radius. Certainly, this is not an exhaustive
list of references, however, we do not know of a solution to this
problem where the underlying geometry fails to satisfy such bounds.
In \cite{BMc} \emph{harmonic coordinates}
are used to  obtain uniform control at a small scale on the manifold,
making it explicit how the assumption of such 
bounds assist in the proof.

As a consequence, McIntosh has asked whether this assumption
is a necessary condition. In this section, we demonstrate that
it is not. We first construct very simple
$2$-dimensional geometries which fail these bounds
as suggested by Anton Petrunin and Sergei V. Ivanov.
We demonstrate by application of previous results that
the Kato square root problem can be solved on these geometries.
Furthermore, we  demonstrate that such examples are
abundant in all dimensions above $2$. 

\subsection{Cones as Lipschitz graphs}
\label{Sect:nCone}

Cones and their smooth counterparts will 
be instrumental to the results we obtain in this section.
We begin by examining
$n$-cones and establishing some of
their properties that will be of use to us.

Let $h, r > 0$ and consider the $n$-cone
in $\R^{n+1}$ given by 
$$\cC_{r,h}^n = \set{(x,t) \in \R^{n+1}: \modulus{x} = \frac{r}{h} (h - t),\ t \in [0, h]}.$$
This is a cone of height $h$ and radius $r$.
Letting $H:B_r(0) \to \R$ be the height function
given by 
$$H_{r,h}(x) = h\cbrac{1 - \frac{\modulus{x}}{r}},$$
we note that the cone can be realised as
the image of the graph function $F_{r,h}(x) = \Graph H(x) = (x, H_{r,h}(x))$.

Let $U$ be an open set in $\R^n$ such that $B_r(0) \subset U$.
Then, define $G_{r,h}: U \to \R^{n+1}$
as the map $F_{r,h}$ whenever $x \in B_r(0)$
and $(x,0)$ otherwise. First, we prove the following. 

\begin{proposition}
The map $G_{r,h}$ satisfies
$$ \modulus{x - y} 
	\leq \modulus{G_{r,h}(x) - G_{r,h}(y)} 
	\leq \sqrt{1 + \frac{h^2}{r^2}} \modulus{x - y}.$$
\end{proposition}
\begin{proof}
Suppose $x, y \not\in B_r(0)$. Then, it is easy to see
that $\modulus{G_{r,h}(x) - G_{r,h}(y)} = \modulus{x - y}$.
When $x, y \in B_r(0)$, then
$G_{r,h}(x) - G_{r,h}(y) = (x - y, hr^{-1}(\modulus{y} - \modulus{x}))$
and by the reverse triangle inequality $\modulus{\modulus{x} - \modulus{y}} \leq \modulus{ x - y}$
we obtain the inequality in the conclusion.

For the remaining case, let $x \in B_r(0)$
and $y \not\in B_r(0)$. Then, 
$G_{r,h}(x) - G_{r,h}(y) = (x - y, h(1 - \modulus{x}r^{-1}))$. 
But by choice of $x$ and $y$, 
$r \leq \modulus{y} \leq\modulus{x - y} + \modulus{x}$
which implies that 
$$ 1 - \frac{\modulus{x}}{r} \leq \frac{\modulus{x - y}}{r}$$
and the desired estimate follows by direct calculation.
\end{proof}

This proposition tells us that, in particular, the
map $G_{r,h}$ is almost-everywhere differentiable. 
Therefore, we define
the pullback metric $\mg = \pullb{G}_{r,h}\inprod{\mdot,\mdot}_{\R^{n+1}}$
on $U$. The following proposition is immediate. 
\begin{proposition}
\label{Prop:ConeClose}
Let $\gamma: I \to U$ be a smooth curve 
such that $\gamma(0) \not \in \set{0} \union \bnd B_r(0)$.
Then, 
$$ \modulus{\gamma'(0)} \leq \modulus{(G_{r,h} \comp \gamma)'(0)} 
	\leq \sqrt{1 + \frac{h^2}{r^2}} \modulus{\gamma'(0)}.$$
Moreover, for $u \in T_x U$, $x \not\in \set{0} \union \bnd B_r(0)$
(and in particular for almost-every $x$), 
$$ \modulus{u}_{\delta} \leq \modulus{u}_{\mg} \leq 
\sqrt{1 + \frac{h^2}{r^2}} \modulus{u}_{\delta},$$
where $\delta$ is the usual inner product on $U$
induced by $\R^n$. 
\end{proposition}

A particular consequence of this observation is
 that the metrics $\mg$ and $\delta$
are $\sqrt{1  + (hr^{-1})^2}$-close on $U$.

\subsection{Two dimensional examples}

In this section, we construct
$2$-dimensional examples to illustrate
that the Kato square root problem can be solved for
metrics with zero injectivity radius.
First, we establish some basics facts about $2$-cones.
We first realise the $2$-cone as a quotient space
since this makes it easier for us to study its geodesics.

Let
$\cT_{r,h}$ 
denote the region enclosed by the sector of
angle $\theta = \frac{2\pi r}{\sqrt{h^2 + r^2}}$
inside the ball 
$B_{\sqrt{h^2 + r^2}}(0,S)$
where $S = \sqrt{h^2 + r^2}\cos\cbrac{\frac{\pi r}{\sqrt{h^2 + r^2}}}$
containing the triangular region
$$\set{(x,t) \in \R^2: \modulus{x}\leq \frac{R}{S}(S - t),\ t \in [0,S]},$$
of base length $2R$ and height $S$, with 
$R = \sqrt{h^2 + r^2}\sin\cbrac{\frac{\pi r}{\sqrt{h^2 + r^2}}}$
as illustrated in Figure \ref{Fig:Sector}.
\begin{figure}[H]
% can put H instead of h! to force position, but requires \usepackage{float}
\includegraphics[scale=0.7]{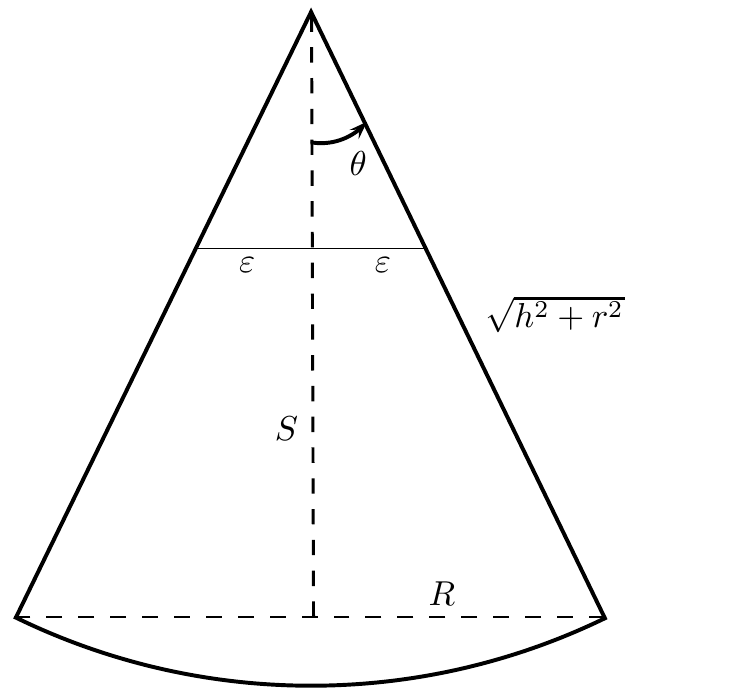}
\caption{The sector $\cT_{r,h}$.}
\label{Fig:Sector} 
\end{figure}
Define the map $q_{r,h}: \cT_{r,h} \to \R^2$
as the identity map in $\interior \cT_{r,h}$, the interior of $\cT_{r,h}$, and
by identifying the points  
$(R ( 1 - t S^{-1}), t)$ with 
$(-R (1 - t S^{-1}), t)$.

\begin{lemma}
The quotient space $\cT_{r,h}/{q_{r,h}}$ is 
isometric to $\cC_{r,h}^2$.
\end{lemma}
\begin{proof}
Intersect the cone $\cC_{r,h}^2$ with a $2$-plane
containing the $z$-axis in $\R^3$. This produces a
triangle whose height is $h$ and radius $r$ 
with hypotenuse $\sqrt{h^2 + r^2}$. This along
with the fact that the base of the cone has
circumference $2\pi r$ allows us to compute
$S$, $R$ and $\theta$ as in our definition above.
The identification is then obvious.
\end{proof} 

Given this, we observe the following about
geodesics of $\cC_{r,h}^2$. 

\begin{lemma}
The straight lines inside $\cT_{r,h}$ between points
map to geodesics in $\cC_{r,h}^2$ under $q_{r,h}$.
\end{lemma}
\begin{proof}
The map $q_{r,h}$ bends $\cT_{r,h}$ to the cone $\cC^2_{r,h}$. That is, it
is an isometry. 
\end{proof}

The next is a very important lemma which establishes
the existence of non-unique geodesics of decreasing length. 

\begin{lemma}
\label{Lem:2ConeGeo}
Given $\epsilon > 0$, there exists two points $x, x'$
and distinct minimising smooth geodesics $\gamma_{1,\epsilon}$ and $\gamma_{2,\epsilon}$
between $x$ and $x'$ of length $\epsilon$.
Furthermore, there are two constants $C_{1,r,h,\epsilon}, C_{2,r,h,\epsilon} > 0$ depending
on $h,\ r$ and $\epsilon$ such that 
the geodesics $\gamma_{1,\epsilon}$ and $\gamma_{2,\epsilon}$
are contained in $G_{r,h}(A_\epsilon)$ where
$A_\epsilon$ is the Euclidean annulus 
$$\set{x \in B_r(0): C_{1,r,h,\epsilon} 
	< \modulus{x} < C_{2,r,h,\epsilon}}.$$
\end{lemma}
\begin{proof}
As a consequence of the previous lemma,
it suffices to realise $\gamma_{1,\epsilon}$ and $\gamma_{2,\epsilon}$ as 
straight lines of length $\epsilon$ in $\cT_{r,h}$. This is easy. 
If we solve for $t$ in  
$R ( - tS^{-1}) = \epsilon$, we find that
$t = (1 - \frac{\epsilon}{R})$. Thus,
fix the point $y = (0,S(1 - \frac{\epsilon}{R}))$. 
This can be identified with the point $x$. 
Now consider $y_\pm = (\pm \epsilon, S(1 - \frac{\epsilon}{R})$. 
These two points are identified under the
quotient map $q_{r,h}$, so identify this with $x'$. Then,
we can take the straight line segment $l_{+}$ from $y$ to $y_{+}$
and identify this with $\gamma_{1,\epsilon}$. Then $\gamma_{2,\epsilon}$
can be obtained by the straight line $l_{-}$ from $y$ to $y_{-}$.
Since $q_{r,h}$ is an isometry, $\gamma_{1,\epsilon}$ and $\gamma_{2,\epsilon}$
each have length $\epsilon$.

Next, consider the ball of radius $\tau$
centred $(0,h)$, the apex of the region
region $\cT_{r,h}$. It is easy to verify that 
the intersection
$B_{\tau}(0,h) \intersect \cT_{r,h}$ 
correspond to balls $B_{\tau}(0) \subset \R^2$
under quotienting and via the inverse of 
$G_{r,h}$ which is the projection map.
It is easy to see that the lines $l_{\pm}$
are contained in $B_{C_{1,r,h,\epsilon}}(0, h) \intersect B_{C_{2,r,h,\epsilon}}(0,h)$
where $\C_i$ are constants dependent 
on the identification, $r,\ h$ and $\epsilon$. 
Thus, it follows that 
$\gamma_{1,\epsilon}, \gamma_{2,\epsilon}: I \to A_{\epsilon}$.
\end{proof}

We first demonstrate that the Kato square root problem
can be solved for a non-smooth metric with zero injectivity 
radius.

\begin{theorem}
For any $C > 1$, there exists a metric
$\mg$ which is $C$-close to the Euclidean
metric $\delta$ obtained via a Lipschitz pullback 
for which $\inj(\R^2,\mg) = 0$.
Furthermore, the Kato square root problem
\eqref{Ass:KatoFn} can be solved
for $(\R^2, \mg)$ under the ellipticity
assumptions \eqref{Ass:Ell}. 
\end{theorem}
\begin{proof}
Set $r = 1$ and choose $h > 0$ such that $\sqrt{1 + h^2} = C$
and apply  Proposition \ref{Prop:ConeClose} with $U = \R^2$.
As a consequence of Lemma \ref{Lem:2ConeGeo},
there exist distinct minimising-geodesics $\gamma_{1, \epsilon}$ $\gamma_{2,\epsilon}$
between two points $x_{\epsilon}$, $x_{\epsilon}'$ for each $\epsilon > 0$. 
Therefore, $\inj(\R^2, \mg, x_\epsilon) \leq \epsilon$ and hence
$\inj(\R^2,\mg) = 0$.
The fact that the Kato square root problem \eqref{Ass:KatoFn}
has solutions is immediate from Corollary \ref{Cor:KatoFn}. 
\end{proof}

Using a similar construction as we did in the previous
example, we produce a smooth metric with
zero injectivity radius solving the aforementioned problem.

\begin{theorem}
\label{Thm:KatoInj2dim}
For any $C > 1$, there exists a smooth 
metric which is $C$-close to $\delta$ 
on $\R^2$ for which $\inj(\R^2,\mg) = 0$.
The Kato square root problem \eqref{Ass:KatoFn}
then has a solution on $(\R^2, \mg)$ under
\eqref{Ass:Ell}. 
\end{theorem}
\begin{proof}
As in the proof of the previous theorem, 
set $r = 1$ and choose $h$ 
such that $\sqrt{1 + h^2} = C$. 
Consider the $2$-cone of radius
$1$ contained in $B_{2}(0)$.
Then, for any $k > 0$, we can find an annulus $A_{\frac{1}{n}}$
containing two points $x_k$ and $x_k'$ and 
non-unique minimising geodesics $\gamma_{1,k}$ and $\gamma_{2,k}$
between them of length $\frac{1}{k}$ by  Lemma
\ref{Lem:2ConeGeo}. As a consequence, 
we can smooth the cone at the apex
above $A_{\frac{1}{k}}$ and below at
the base to obtain a smooth map
$G^\infty_{1,h,k}:B_{2}(0) \to \R^{n+1}$.
These geodesics are still geodesics
in this new space as $A_{\frac{1}{k}}$ is an open 
set and hence, it is totally geodesic.
Now, choose a set of points $y_k = (0,3k) \in \R^2$
for $k \geq 0$ and 
consider the map:
$$
G(x) = \begin{cases} 
		x & x\not\in B_{2}(y_k) \\
		G^\infty_{1,h,k}(x - y_k) &x \in B_2(y_k).
	\end{cases}$$
It is easy to see that $G$ is a smooth map 
and pullback metric $\mg = \pullb{G}\inprod{\mdot,\mdot}_{\R^{n+1}}$
is smooth and $C$-close to the Euclidean
metric. 
By construction,
$\inj(\R^2,\mg, x_k) \leq \frac{1}{k}$ and hence, $\inj(\R^2, \mg) = 0$.
As before, the fact that the Kato square root problem \eqref{Ass:KatoFn}
has solutions is immediate from Corollary \ref{Cor:KatoFn}. 
\end{proof}

\subsection{Higher dimensions}

In this section, we demonstrate that 
metrics with zero injectivity radius
are abundant in the space of rough metrics. 
We use this 
to show that the Kato square root problem can be solved
for a wide class of metrics and also demonstrate
that there exist low regularity metrics
on compact manifolds for which injectivity radius
bounds fail.
Throughout this section, we assume that 
the dimension is at least $2$.

First, as in the previous section, 
we define the annulus $A_\epsilon^n$ in our more general
situation as
$$A_\epsilon^n = \set{x \in B_r(0): C_{1, r, h, \epsilon} 
		< \modulus{x} < C_{2, r, h, \epsilon}},$$
where the constants $C_{1,r,h,\epsilon}$ and $C_{2,r,h,\epsilon}$
are the ones guaranteed by Lemma \ref{Lem:2ConeGeo}.
We then obtain the following natural generalisation of this lemma. 

\begin{lemma}
\label{Lem:nConeGeo}
For every $\epsilon > 0$, the set
$G_{r,h}(A_\epsilon^n)$ contains 
two points $x_\epsilon$ and $x_\epsilon'$ and
two distinct minimising geodesics $\gamma_{1,\epsilon}$
$\gamma_{2, \epsilon}$ between these two points
such that their length is $\epsilon$.
\end{lemma}
\begin{proof}
Let $x = (x_1, \dots, x_{n}) \in \Sph^{n-1} \subset \R^n$ 
be points on the $(n-1)$-sphere.
Then, let $s(x) = (x_1, x_2, -x_3, \dots, -x_n)$.
Note that the fixed point set  of this map $s$ precisely
the circle $\Sph^1$.

Now,  for points $(x,t) \in \cC_{r,h}^n$, define 
$\Phi(x,t) = (s(x), t).$ The map $\Phi:\cC_{r,h}^n \to \cC_{r,h}^n$
is an isometry. Moreover, the restricted
map $\Phi_\epsilon(x,t) = \Phi\rest{A_\epsilon^n}: A_\epsilon^n \to A_\epsilon^n$
is also an isometry. It is easy to see that 
$A_\epsilon^2 = \set{(x,t): \Phi_\epsilon(x,t) = (x,t)}$.
Since $A_\epsilon^n$ stays away from the base and apex of $\cC_{r,h}^n$, 
it is a smooth manifold submanifold of $\R^{n+1}$ and 
by the fact that $\Phi_\epsilon$ is an isometry, 
Theorem 1.10.15 in \cite{Klingenberg} guarantees us
that  $A_\epsilon^2$ is a totally geodesic submanifold.
Invoking Lemma \ref{Lem:2ConeGeo} completes the proof.
\end{proof}

The following final lemma is
instrumental in proving the 
main theorem of this section.
Note that we write $B_{r}^\ast(x) = \psi^{-1}(B_r(\psi(x)))$
inside a coordinate chart $(U,\psi)$.

\begin{lemma}
\label{Lem:GCover}
Let $\cM$ be a smooth manifold with a continuous
metric $\mg$. Then, for any $x_0$ and $C > 1$, 
there exists an at most countable
collection $\sC = \set{(U_i,\psi_i)}$ of charts covering
$\cM$ and an $r_0 > 0$ such that 
\begin{enumerate}[(i)]
\item for all $y \in U_i$ and $u \in T_y U_i$, 
	$$C^{-1} \modulus{u}_{\pullb{\psi_i}\delta} 
		\leq \modulus{u}_\mg 
		\leq C \modulus{u}_{\pullb{\psi_i}\delta},$$
\item $\close{B_{r_0}^\ast(x_0)} \subset U_0$, and
\item $\close{B_{r_0}^\ast(x_0)} \intersect U_i = \emptyset$
	for $i > 0$.
\end{enumerate} 
\end{lemma}
\begin{proof}
For $x \in \cM$, we can find a chart 
$(V_x, \psi_x)$ satisfying the inequality
$C^{-1} \modulus{u}_{\pullb{\psi_i}\delta}
		\leq \modulus{u}_\mg 
		\leq C \modulus{u}_{\pullb{\psi_i}\delta}$
for all $u \in T_yV_x$ by the continuity of $\mg$.

Let $\set{V_i}$ be a countable subcover, choosing
the index so that $x_0 \in V_0$.
We can choose $r_0 > 0$ small such that
$\close{B_{r_0}(\psi_0(x_0))} \subset \psi_0(V_0)$. 
Define a new set of charts by restricting the
sets $V_i$ on setting $\tilde{U}_0 = V_0$ 
and $\tilde{U}_i = V_i\setminus \close{B_{r_0}^\ast(x_0)}$
for $i > 0$. Then, define 
$\sC = \set{\tilde{U}_i: \tilde{U}_i \neq \emptyset}$.

It is easy to see that $\sC$ covers $\cM$. 
%but we 
%demonstrate this for the sake of completeness.
%Fix $x \in \cM$. If $x \in \close{B_{r_0}(x_0)}$
%then $x \in V_0$. Otherwise, $x \in V_i$
%for some $i$ and by choice of $x$,
%we have that $x \in V_i \setminus \close{B_{r_0}(x_0)} \neq \emptyset$.
%Thus, $x \in \tilde{U_i} \in \sC$.
Since the $\tilde{U}_i$ are obtained
as restrictions of the $V_i$, 
it is easy to see that the desire inequality in (i)
still holds. This shows (i).
That (ii) is true is immediate, and (iii) 
is true by the construction of the $\tilde{U}_i$.
\end{proof}

Whit the aid of this tool, we prove the following
main theorem of this section.

\begin{theorem}
\label{Thm:KatoInj}
Let $\cM$ be a smooth manifold of dimension at least $2$ 
and $\mg$ a continuous metric. 
Given $C >1$, and a point $x_0 \in \cM$, there exists a rough metric $\mh$
such that: 
\begin{enumerate}[(i)]
\item it preserves the topology of $\cM$,
\item it is smooth everywhere but at $x_0$,
\item the geodesics through $x_0$ are Lipschitz,
\item it is  $C$-close to $\mg$,
\item $\inj(\cM \setminus\set{x_0},\mh) = 0$. 
\end{enumerate} 
\end{theorem}
\begin{proof}
For a given $x_0 \in \cM$
take the cover given by Lemma \ref{Lem:GCover}
with constant $C_1 > 1$ to be chosen later.
Let $\set{\phi_i}$ be a smooth partition of unity
subordinate to $\sC = \set{U_0, U_1, \dots}$. 
Note that since $\close{B_{r_0}^\ast(x_0)} \intersect U_i = \emptyset$
for $i > 0$, we have that 
$\phi_0 \equiv 1$ on $\close{B_{r_0}^\ast(x_0)}$.

Define a new metric 
$\mgt = \sum_{i} \phi_i \pullb{\psi}_i \delta$. It is easy to 
see that this metric is smooth and a direct calculation 
at an arbitrary point $x$ will show that it is $C_1$-close
to $\mg$.

Set $r = \frac{r_0}{10}$, and choose $h > 0$ such that 
$\sqrt{1 + \frac{h^2}{r^2}} \leq C_2$, where $C_2 > 0$
to be chosen later. Inside $\phi_0(U_0)$, remove the 
Euclidean ball $B_r(y_0) \subset \R^n$ where $y_0 = \psi_0(x_0)$
and 
affix an $n$-cone of radius $r$ and height $h$
via the map $G_{r,h}:B_{r_0}(\phi(x_0)) \to \R^{n+1}$
on setting $U = B_{r_0}(\phi(x_0))$ in our construction
in \S\ref{Sect:nCone}.
Note that 
$\pullb{G}_{r,h}\inprod{\mdot, \mdot}_{\R^{n+1}} = \inprod{\mdot,\mdot}_{\R^n}$
in the annulus $B_{r_0}(y_0) \setminus \close{B_{r}(y_0)}$.
We can smooth the base of the cone to produce a map 
$\tilde{G}_{r,h}:B_{r}(y_0) \to \R^{n+1}$ such that 
there exists $\tau \in (0, r)$ with 
$\pullb{\tilde{G}}_{r,h}\inprod{\mdot,\mdot}_{\R^{n+1}} = \inprod{\mdot,\mdot}_{\R^n}$
for $B_{r}(y_0) \setminus \close{B_{r - \tau}(y_0)}$,
a $c \in (0,1)$
with $\tilde{G}_{r,h} = G_{r,h}$ in $B_{cr}(y_0)$, 
and so that the Lipschitz constant $C_2$ is unaltered.

Now, consider the pullback metric $\mgt_{0} = \pullb{\tilde{G}_{r,s}}$
on $B_{r_0}(y_0)$. It is easy to see that this metric 
preserves the topology in $B_{r_0}(y_0)$. 
Thus, define
$$\mh(x) = \begin{cases} 
	\pullb{\psi_0}\mgt_0(x) &x \in B_{r_0}^\ast(x_0). \\
	\mgt(x)			&\text{otherwise}.
	\end{cases}
$$
This is a smooth metric away from $x_0$, because by what we have said before,
$\phi_0 \equiv 1$ on the ball $B_{r_0}^\ast(x_0)$, 
$\mgt_0 = \delta$ in $\B_{r_0}(y_0) \setminus \close{B_{r - \tau}(y_0)}$, 
and because $\mgt_0$ fails to be smooth only at $y_0$.
As we have already mentioned, $\tilde{G}_{r,h} = G_{r,h}$ in a neighbourhood 
of the apex, and hence, it is a Lipschitz map there. As a consequence,
geodesics through $x_0$ are Lipschitz.

So far, we have shown (i) to (iii). To show (iv), 
note that $\mh$ is $C_2$-close to $\mgt$ and $\mgt$
is $C_1$ close to $\mg$, we have that 
$\mh$ is $C_1C_2$-close to $\mg$.
On setting $C_1 = C_2 = \sqrt{C}$, we obtain
that $\mg$ is $C$-close to $\mh$.

To show that $\inj(\cM,\mh) = 0$, note that 
there is some $\epsilon_0$ so that
whenever $\epsilon  \in (0, \epsilon_0)$, 
$\tilde{G}_{r,h}(A_\epsilon) = G_{r,h}(A_\epsilon)$.
The set $A_\epsilon$ is an open set in $\R^n$
and hence, $\tilde{G}_{r,h}(A_\epsilon)$
it is an open set in $\img \tilde{G}_{r,h}$. 
Open sets are totally geodesic submanifolds,
and by Lemma \ref{Lem:nConeGeo}, $G_{r,h}(A_\epsilon)$ contains a two distinct minimising geodesics of
length $\epsilon$ between points $y_\epsilon$ and $y_\epsilon'$. Thus, 
on setting $x_\epsilon = \psi_0^{-1}(G_{r,h}^{-1}(x_\epsilon))$, 
we have that
$\inj(\cM, \mh, x_\epsilon) \leq \epsilon$.
Therefore, $\inj(\cM,\mh) = 0$.
\end{proof} 

As aforementioned,
 we are currently able to prove the Kato square
root problem when $\mg$ is complete, smooth,
$\modulus{\Ric} \leq \eta$ and $\inj(\cM,\mg) \geq \kappa > 0$.
However, by this theorem, we are able to find arbitrarily 
close metrics $\mh$ to $\mg$ for which the 
injectivity radius
bounds fail, and by 
Corollary \ref{Cor:KatoFn},
we can solve the Kato square
root problem for such metrics.
This leads us to believe that the
lower bounds on injectivity radius
in the proofs
of the Kato square root problem is a
technical assumption.

\def\cprime{$'$}
\providecommand{\bysame}{\leavevmode\hbox to3em{\hrulefill}\thinspace}
\providecommand{\MR}{\relax\ifhmode\unskip\space\fi MR }
% \MRhref is called by the amsart/book/proc definition of \MR.
\providecommand{\MRhref}[2]{%
  \href{http://www.ams.org/mathscinet-getitem?mr=#1}{#2}
}
\providecommand{\href}[2]{#2}

\end{document}